\documentclass[12pt,reqno]{amsart}
\usepackage{a4wide}

\numberwithin{equation}{section}

\newtheorem{theorem}{Theorem}[section]
\newtheorem{proposition}[theorem]{Proposition}

\newtheorem{lemma}[theorem]{Lemma}

\theoremstyle{definition}

\theoremstyle{remark}
\newtheorem{remark}[theorem]{Remark}

\newcommand{\ep}{\varepsilon}
\newcommand{\Om}{\Omega}

\newcommand{\de}{\delta}

\begin{document}

\title[  Regularization of point vortices
 ]
{
   Regularization of point vortices  for the Euler equation in dimension two
 }
 \author{Daomin Cao
}

\address{Institute of Applied Mathematics, Chinese Academy of Science, Beijing 100190, P.R. China}

\email{dmcao@amt.ac.cn}

 \author{Zhongyuan Liu
}

\address{Institute of Applied Mathematics, Chinese Academy of Science, Beijing 100190, P.R. China}

\email{liuzy@amss.ac.cn}

\author{Juncheng Wei
}

\address{Department of Mathematics, The Chinese University of Hong Kong, Shatin, N.T., Hong Kong }

\email{wei@math.cuhk.edu.hk}

\begin{abstract}
In this paper, we construct  stationary classical solutions of the
incompressible Euler equation approximating singular stationary
solutions of this equation.
 This procedure is carried out by constructing solutions to the
 following elliptic problem
\[
 \begin{cases}
-\ep^2 \Delta u=\sum_{j=1}^m \chi_{\Om_j}(u-q-\frac{\kappa_j}{2\pi}\ln\frac{1}{\ep})_+^p, \quad & x\in\Omega, \\
u=0, \quad & x\in\partial\Omega,
\end{cases}
\]
where $p>1$,  $\Omega\subset\mathbb{R}^2$ is a bounded domain,
$\Om_i\subset\subset\Om, i=1\cdots,m$ are suitable small domains
such that $\Om_i\bigcap\Om_j $ is empty if $i\neq j$ and $q$ is a
harmonic function.

We showed that if $\Omega$ is simply-connected smooth domain, then
for any given stable critical point of Kirchhoff-Routh function
$\mathcal{W}(z_1,\cdots,z_m)$ with the positive strength
$\kappa_i>0$, there is a
 stationary classical solution approximating stationary $m$ points vortex solution of incompressible Euler
equations with vorticity $\sum_{j=1}^m\kappa_i$.

Existence and asymptotic behavior of single point non-vanishing
vortex solutions were studied by D. Smets and J. Van Schaftingen in
\cite{SV}.
  \\[12pt]
 \emph{AMS 2000 Subject Classifications: Primary $35\mathrm{J}60$;\  Secondary   $35\mathrm{JB}05$;  $35\mathrm{J}40$
      }
\end{abstract}

\maketitle

\section{Introduction and main results}

The incompressible Euler equations
\begin{equation}\label{1.2}
\begin{cases}
 \mathbf{v}_t+(\mathbf{v}\cdot\nabla)\mathbf{v}=-\nabla P,\\
\nabla\cdot\mathbf{v}=0,
\end{cases}
\end{equation}
describe the evolution of the velocity $\mathbf{v}$ and the pressure
$P$ in an incompressible flow. In $\mathbb{R}^2$, the vorticity of
the flow is defined by $\omega=\nabla\times\mathbf{v}:=\partial_1
v_2-\partial_2 v_1$, which satisfies the equation
\[
\omega_t+\mathbf{v}\cdot\nabla\omega=0.
\]

Suppose that $\omega$ is known, then the velocity $\mathbf{v}$ can
be recovered by Biot-Savart law as following:
\[
\mathbf{v}=\omega\,*\frac{1}{2\pi}\frac{-x^\bot}{|x|^2},
\]
where $x^\bot=(x_2,\,-x_1)$ if $x=(x_1,\,x_2)$. One special singular
solutions of Euler equations is given by
$\omega=\sum^m_{i=1}\kappa_i\delta_{x_i(t)}$, which is related
\[
\mathbf{v}=-\sum^m_{i=1}\frac{\kappa_i}{2\pi}\frac{(x-x_i(t))^\bot}{|x-x_i(t)|^2}.
\]
and the positions of the vortices $x_i: \mathbb{R}\rightarrow
\mathbb{R}^2$ satisfy the following Kirchhoff law:
\[
\kappa_i\,\frac{dx_i}{dt}=(\nabla_{x_i}\mathcal{W})^\bot
\]
where $\mathcal{W}$ is the so called Kirchhoff-Routh function
defined by
\[
\mathcal{W}(x_1,\cdots, x_m)=\frac{1}{2}\sum_{i\neq
j}^m\frac{\kappa_i\kappa_j}{2\pi}\log\frac{1}{|x_i-x_j|}.
\]

In simply-connected bounded domain $\Omega\subset \mathbb{R}^2$,
similar singular solutions also exist. Suppose that the normal
component of $\mathbf{v}$ vanishes on $\partial\Omega$, then the
Kirchhoff-Routh function is

\begin{equation}\label{W1}
\mathcal{W}(x_1,\cdots, x_m)=\frac{1}{2}\sum_{i\neq
j}^m{\kappa_i\kappa_j}G(x_i,\,x_j)+
\frac{1}{2}\sum_{i=1}^m{\kappa_i^2}H(x_i,\,x_i),
\end{equation}
where $G$ is the Green function of $-\Delta$ on $\Omega$ with 0
Dirichlet boundary condition and $H$ is its regular part (the Robin
function).

 Let $v_n$ be the outward component of the velocity $\mathbf{v}$ on
 the boundary $\partial\Omega$, then we see that $\int_{\partial\Omega}v_n=0$ due to
 the fact that $\nabla\cdot\mathbf{v}=0$. Suppose that
 $\mathbf{v}_0$ is the unique harmonic field whose normal component
 on the boundary $\partial\Omega$ is $v_n$, then $\mathbf{v}_0$
 satisfies

\begin{equation}\label{v0}
\begin{cases}
 \nabla\cdot\mathbf{v}_0=0,\,\,\text{in}\,\Omega,\\
\nabla\times\mathbf{v}_0=0,\,\,\text{in}\,\Omega,\\
n\cdot\mathbf{v}_0=v_n,\,\, \text{on}\,\partial\Omega.
\end{cases}
\end{equation}

If $\Omega$ is simply-connected, then $\mathbf{v}_0$ can be written
$\mathbf{v}_0=(\nabla\psi_0)^\bot$, where the stream function
$\psi_0$ is determined up to a constant by
\begin{equation}\label{psi}
\begin{cases}
 -\Delta \psi_0=0,\,\,\text{in}\,\Omega,\\
-\displaystyle\frac{\partial\psi_0}{\partial\tau}=v_n,\,\,
\text{on}\,\partial\Omega,
\end{cases}
\end{equation}
where $\frac{\partial\psi_0}{\partial\tau}$ denotes the tangential
derivative on $\partial\Omega$. The Kirchhoff-Routh function
associated to the vortex dynamics becomes(see Lin \cite{Lin})

\begin{equation}\label{KR}
\mathcal{W}(x_1,\cdots,x_m)=\frac{1}{2}\sum_{i\neq
j}^m\kappa_i\kappa _jG(x_i,x_j)+\frac{1}{2}\sum^{m}_{i=1}\kappa^2
_iH(x_i,x_i)+\sum^{m}_{i=1}\kappa_i\psi_0(x_i).
\end{equation}

 It is known that critical points of the Kirchhoff-Routh function $\mathcal{W}$
 give rise to stationary vortex points solutions of the Euler
 equations. As for the existence of critical points of $\mathcal{W}$ given by \eqref{W1},
 we refer to \cite{BT}.

 Roughly speaking, there are two methods to construct stationary
solutions of the Euler equation, which are the vorticity method and
the stream-function method. The vorticity method was first
established by Arnold and Khesin \cite{AK} and further developed by
Burton \cite{B} and Turkington \cite{T}.

The stream-function method consists in observing that if $\psi$
satisfies  $-\Delta \psi=f(\psi)$ for some function $f\in
C^1(\mathbb{R})$, then $\mathbf{v}=(\nabla\psi)^\bot$ and
$P=F(\psi)-\frac{1 }{2}|\nabla\psi|^2$ is a stationary solution to
the Euler equations, where
$(\nabla\psi)^\bot:=(\frac{\partial\psi}{\partial
x_2},-\frac{\partial\psi}{\partial x_1}), F(t)=\int_0^tf(s)ds$.
Moreover, the velocity $\mathbf{v}$ is irrotational on the set where
$f(\psi)=0$.

Set $q=-\psi_0$ and $u=\psi-\psi_0$, then $u$ satisfies the
following boundary value problem
\begin{equation}\label{1.3}
\begin{cases}
-\Delta u=f(u-q),\quad & x\in\Omega,\\
u=0,\quad &x\in\partial\Omega.
\end{cases}
\end{equation}
 In addition, if we suppose that $\inf_\Omega q>0$ and
$f(t)=0,~t\leq0$, the vorticity set $\{x: f(\psi)>0\}$ is bounded
away from the boundary.

The motivation to study \eqref{1.3} is to justify the weak
formulation for point vortex solutions of the incompressible Euler
equations by approximating these solutions with classical solutions.

 Marchioro and Pulvirenti \cite{MP} have approximated these solutions
 on finite time intervals by considering regularized initial data
 for the vorticity. On the other hand, the stationary point vortex
 solutions can also be approximated by stationary classical
 solutions. See e.g. \cite{BF1,N,SV,T,Y1} and the references therein. It is worth
 pointing out that the above approximations can just give
 explanation for the formulation to single point vortex solutions.
 In this paper, we will show that multi-point vortex solutions can
 be approximated by stationary classical solutions.
There are many results for problem \eqref{1.3} on the existence and
asymptotic behavior of solutions under various assumptions. In
\cite{Ba,BF1,FB,N,T}, the constrained variation methods were used to
find solutions for the equation
\begin{equation}\label{1.4}
\begin{cases}
-\Delta u=\lambda f(u-q),\quad & x\in\Omega,\\
u=0,\quad &x\in\partial\Omega,
\end{cases}
\end{equation}
under the constraint $\int_\Omega F(u-q)=\mu$,  where $\lambda>0$ is
a Lagrange multiplier a priori unknown. On the other hand, in
\cite{AS,AY,Ni,Y1,Y2}, the solutions were obtained by using Mountain
Pass Lemma for various nonlinearities. For the asymptotic behavior,
Berger and Fraenkel \cite{BF1} began studying the asymptotic
behavior for variable $\mu$ and $q$, but the lack of information
about $\lambda$ is still an obstacle.

To avoid this obstacle, Yang \cite{Y1} studied the minimization of
the functional $\frac{1}{2}\int_\Omega|\nabla
u|^2-\frac{1}{\ep^2}\int_\Omega F(u-q)$ under the natural constraint
$\int_\Omega|\nabla u|^2-\frac{1}{\ep^2}\int_\Omega uf(u-q)=0$ and
obtained the asymptotic behavior of the solutions $u_\ep$ as
$\ep\rightarrow0$ for $\Omega=\mathbb{R}^2_+,~q(x)=Wx_1+d$, where
$W,d>0$. That is, set $A_\ep=\{x\in\mathbb{R}^2_+: f(u_\ep-q)>0\},
\kappa_\ep=\frac{1}{\ep^2}\int_\Omega f(u_\ep-q)$ and $x_\ep\in
A_\ep$, then $diam A_\ep\rightarrow0,~
dist(x_\ep,\partial\mathbb{R}_+^2)\rightarrow0$ and
$\frac{u_\ep}{\kappa_\ep}-G(x_\ep,\cdot)\rightarrow0$ in
$W_{loc}^{1,r}(\mathbb{R}_+^2)$ for $r\in[1,2)$. Later on, similar
results were obtained in \cite{LYY} for bounded domains with
additional information that $q(x_\ep)\rightarrow\min_\Omega q$.
However, it has been pointed out in \cite{SV} that the solutions
obtained above corresponded to desingularization of point-vortex
solutions with vanishing vorticity. To get non-vanishing vortex
solutions, D. Smets and J. Van Schaftingen \cite{SV} investigated
the following problem
\begin{equation}\label{1.1}
\begin{cases}
-\ep^2 \Delta u=\left(u-q-\frac{\kappa}{2\pi}\ln\frac{1}{\ep}\right)_+^{p},  & \text{in}\;\Om,\\
u=0, &\text{on}\; \partial\Om,
\end{cases}
\end{equation}
and gave exact asymptotic behavior and expansion  of the least
energy solution by estimating the upper bounds on the energy. The
solutions for \eqref{0} in \cite{SV} were obtained by finding a
minimizer of the corresponding functional in a suitable function
space, which can only give approximation to a single point
non-vanishing vortex. This method is hard  to obtain multiple
non-vanishing solutions.

 In this paper, we approximate stationary vortex solutions of Euler equations \eqref{1.2} with multiple non-vanishing vorticity
 by stationary classical solutions.

 Our main result concerning \eqref{1.2} is the
 following:

\begin{theorem}\label{nth1}

Suppose that $\Omega\subset \mathbb{R}^2$ is a bounded
simply-connected smooth domain. Let $v_n:
\partial\Omega\rightarrow \mathbb{R}$ be such that $v_n\in
L^s(\partial\Omega)$ for some $s>1$ satisfying
$\int_{\partial\Omega}v_n=0$. Let $\kappa_i>0,~i=1,\cdots,m$. Then,
for any given stable critical point $(x_1^*,\cdots,x_m^*)$ of
Kirchhoff-Routh function $\mathcal{W}(x_1,\cdots,x_m)$ defined by
\eqref{KR}, there exists $\ep_0>0$, such that for each $\ep\in
(0,\ep_0)$, problem \eqref{1.2} has a stationary solution
$\mathbf{v}_\ep$ with outward boundary flux given by $v_n$, such
that its vorticities $\omega_\ep$ satisfying
\[
supp (\omega_\ep)\subset\cup_{i=1}^m
B(x_{i,\,\ep},C\ep)~~\text{for}~~x_{i,\,\ep}\in\Omega,
~~i=1,\cdots,m
\]
and as $\ep\rightarrow0$
\[
\int_\Omega \omega_\ep\rightarrow \sum_{i=1}^m\kappa_i,
\]

\[
(x_{1,\,\ep},\cdots,x_{m,\,\ep})\rightarrow (x_1^*,\cdots,x_m^* ).
\]
\end{theorem}

\begin{remark} The simplest case, corresponding
 to a single point vortex ( $m=1$ ) was studied by Smets and Van Schaftingen
 \cite{SV} by minimizing the corresponding energy functional. In their paper
 $\mathcal{W}(x_{1,\,\ep})\rightarrow
 sup_{x\in\Omega}\mathcal{W}(x)$. Even in the case $m=1$, our result
 extends theirs to general critical points (with additional assumption that
 the critical point is
 non-degenerate). The method used in \cite{SV} can not be applied to deal with
 general critical point cases. The method used here is constructive and  is completely
 different from theirs.
\end{remark}

\begin{remark}
In this case that $m=1$ suppose that $x_1$ is a strict local
maximum(or minimum) point of Kirchhoff-Routh function
$\mathcal{W}(x)$ defined by \eqref{KR}, statement of Theorem
\ref{nth1} still holds which can be proved similarly(see Remark
\ref{re1.4}). Thus we can obtain corresponding existence result in
\cite{SV}.
\end{remark}

Theorem \ref{nth1} is proved via considering the following problem
\begin{equation}\label{0}
\begin{cases}
-\ep^2 \Delta u=\sum_{j=1}^m \chi_{\Om_j}(u-q-\frac{\kappa_j}{2\pi}\ln\frac{1}{\ep})_+^p, \quad & x\in\Omega, \\
u=0, \quad & x\in\partial\Omega,
\end{cases}
\end{equation}
where $p>1$, $q\in C^2(\Omega)$,  $\Omega\subset\mathbb{R}^2$ is a
bounded domain, $\Om_j\subset\Omega$ is a subdomain such that
$x^*_j\in\Om_j,\,j=1,\cdots,m$ and $\Om_i\cap\Om_j={\O}$ if $ i\neq
j$.

\begin{theorem}\label{th3}

 Suppose $q\in C^2(\Omega)$. For
$\kappa_i>0,~i=1,\cdots,m$. Then, for any given $C^1$-stable
critical point $(x_1^*,\cdots,x_m^*)$ of Kirchhoff-Routh function
$\mathcal{W}(x_1,\cdots,x_m)$ defined by \eqref{KR}, there exists
$\ep_0>0$, such that for each $\ep\in (0,\ep_0)$, \eqref{0} has a
solution $u_\ep$, such that the set  $\Omega_{\ep,i} =\{x:
u_\ep(x)-\frac{\kappa_i}{2\pi}\,\ln\frac{1}{\ep}-q(x)>0\}\subset\subset\Om_i,
\,i=1,2,\cdots,m$ and  each $\Omega_{\ep,\,i}$ shrinks to  $x_i^*\in
\Omega$, as $\ep\to 0$.

\end{theorem}

\begin{remark}\label{re1.4}
For the case $m=1$ suppose that $x_1$ is a strict local maximum(or
minimum) point of Kirchhoff-Routh function $\mathcal{W}(x)$ defined
by \eqref{KR}, statement of Theorem \ref{th3} still holds which can
be proved by making corresponding modification of the proof of
Theorem \ref{th3} in obtaining critical point of $K(z)$ defined by
\eqref{K}(see Propositions 2.3,2.5 and 2.6 \cite{Cao} for detailed
arguments).
\end{remark}

For domains which may not be simply-connected, we show in the
following result that the topology of the domain plays an important
role in the existence of solutions.

\begin{theorem}\label{th1}

Suppose that the homology of  $\Om$  is nontrivial. Then, for any
positive integer $m$, there exists $\ep_0>0$, such that for each
$\ep\in (0,\ep_0)$, \eqref{0} has a solution $u_\ep$, such that the
set  $\Omega_{\ep,i} =\{x:
u_\ep(x)-\frac{\kappa_i}{2\pi}\,\ln\frac{1}{\ep}-q(x)>0\}\subset\subset\Om_i,
\,i=1,2,\cdots,m$ and  each $\Omega_{\ep,\,i}$ shrinks to a point
$x_i^*\in \Omega$, as $\ep\to 0$. Moreover  $x_i^*\neq x_j^*$,  if $
i\neq j$.

\end{theorem}

\begin{remark}
Since $m$ is arbitrary, from Theorem \ref{th1}, we can see that the
number of solutions for  \eqref{1} is unbounded  as $\de\to 0$.
\end{remark}

Not as in \cite{SV} where \eqref{0} is investigated directly,
 we prove Theorem \ref{nth1}, Theorem \ref{th3} and Theorem
 \ref{th1} by considering an equivalent problem of \eqref{0}
 instead. Let $ w=\frac{2\pi}{|\ln\ep|}u$ and
$\delta=\ep(\frac{2\pi}{ |\ln\ep |})^{\frac{p-1}{2}}$, then
\eqref{0} becomes
\begin{equation}\label{1}
\begin{cases}
-\de^2 \Delta w=\sum_{j=1}^m\chi_{\Om_j}\left(w-\kappa_j-\frac{2\pi}{ |\ln\ep|}q(x)\right)_+^{p},  & \text{in}\;\Om,\\
w=0, &\text{on}\; \partial\Om.
\end{cases}
\end{equation}

We will use a reduction argument to prove Theorem~\ref{th3} and
Theorem~\ref{th1}. To
 this end, we need to construct an approximate solution for
\eqref{1}.  For the problem studied in this paper, the corresponding
``limit" problem in $\mathbb{R}^2$ has no bounded nontrivial
solution. So, we will follow the method in \cite{CPY,DY} to
construct an approximate solution. Since there are two parameters
$\de,~\ep$ in problem \eqref{1}, which causes some difficulty, we
must take this influence into careful consideration and give
delicate estimates in order to perform the reduction argument. For
example we need to consider $(s_{1,\de},\cdots,s_{m,\de})$ and
$(a_{1,\de},\cdots,a_{m,\de})$ together in Lemma~\ref{l2.1}.

As a final remark, we point out that problem \eqref{1} can  be
considered as a free boundary problem. Similar problems have been
studied extensively. The reader can refer to \cite{CF,CPY,DY,FW,LP}
for more results on this kind of problem.

This paper is organized as follows.  In section~2, we construct the
approximate solution for \eqref{1}. We will carry out a reduction
argument in section~3 and the main results will be proved in
section~4.  We put some basic estimates in the appendix.

\section{Approximate solutions }

In the section,  we will construct approximate solutions for
\eqref{1}.

Let $R>0$ be a large constant, such that for any $x\in \Om$,
$\Om\subset\subset B_R(x)$. Consider the following problem:

\begin{equation}\label{2.1}
\begin{cases}
-\delta^2\Delta w=( w-a)_+^{p},& \text{in}\; B_R(0),\\
w=0, &\text{on}\;\partial B_R(0),
\end{cases}
\end{equation}
where $a>0$ is a constant. Then, \eqref{2.1} has a unique solution
$W_{\de,a}$, which can be written as

\begin{equation}\label{2.2}
W_{\de,a}(x)=
\begin{cases}
a+\de^{2/(p-1)}s_\de^{-2/(p-1)}\phi\bigl(\frac{|x|}{s_\de}\bigr), &  |x|\le s_\de,\\
a\ln\frac {|x|} R/\ln \frac {s_\de}R, & s_\de\le |x|\le R,
\end{cases}
\end{equation}
where $\phi(x)=\phi(|x|)$ is the unique solution of
\begin{equation*}
-\Delta \phi=\phi^{p},\quad\phi>0,~~\phi\in H_0^1\bigl(B_1(0)\bigr)
\end{equation*}
and $s_\de\in (0,R)$ satisfies
\[
\de^{2/(p-1)}s_\de^{-2/(p-1)}\phi^\prime(1)=\frac{a}{\ln(s_\de/R)},
\]
which implies
\[
\frac{s_\de}{\de|\ln\de|^{(p-1)/2}}\rightarrow\left(\frac{|\phi^\prime(1)|}{a}\right)^{(p-1)/2}>0,\quad\text{as}~~\de\rightarrow0.
\]
Moreover, by Pohozaev identity, we can get that
\[
\int_{B_1(0)}\phi^{p+1}=\frac{\pi(p+1)
}{2}|\phi^\prime(1)|^2~~\text{and}~~\int_{B_1(0)}\phi^{p}=2\pi|\phi^\prime(1)|.
\]

 For any $z \in \Om$,  define
$W_{\de,z,a}(x)=W_{\de,a}(x-z)$. Because $W_{\de,z,a}$
 does not vanish on $\partial\Omega$, we need to  make a  projection. Let
$PW_{\de,z,a}$ be the solution of

\[
\begin{cases}
-\de^2 \Delta w=( W_{\de,z,a}-a)_+^{p},& \text{in } \; \Om,\\
w=0, &\text{on}\; \partial \Om.
\end{cases}
\]
Then

\begin{equation}\label{2.3}
PW_{\de,z,a}= W_{\de,z,a}-\frac a{\ln \frac{R}{s_\de}} g(x,z),
\end{equation}
where $g(x,z)$ satisfies

\[
\begin{cases}
- \Delta g=0,& \text{in } \; \Om,\\
g=\ln\frac{R}{|x-z|}, &\text{on}\; \partial \Om.
\end{cases}
\]
It is easy to see that

\[
g(x,z)=\ln R +2\pi h(x,z),
\]
where $h(x,z)=-H(x,z)$.

 We will
construct solutions for \eqref{1} of  the form

\[
\sum_{j=1}^m PW_{\de,z_{j},a_{\de,j}} +\omega_\de,
\]
where $z_j\in\Omega$ for $j=1,\cdots,m$, $\omega_\de$ is a
perturbation term. To obtain a good estimate for $\omega_\de$, we
need to choose $a_{\de,j}$ properly.

By \eqref{2.3}, we have

\begin{equation}\label{2.4}
\begin{split}
&-\de^2 \Delta  \sum_{j=1}^m PW_{\de,z_j,a_{\de,j}}
-\sum_{j=1}^m\chi_{\Om_j}\left(\sum_{i=1}^mPW_{\de,z_i,a_{\de,i}}-\kappa_j-\frac{2\pi q}{|\ln\ep|}\right)_+^{p}\\
=&  \sum_{j=1}^m \left(
W_{\de,z_j,a_{\de,j}}-a_{\de,j}\right)_+^{p}-
\sum_{j=1}^m\chi_{\Om_j}\left(\sum_{i=1}^m
\left(W_{\de,z_i,a_{\de,i}}-\frac{a_{\de,i}}{\ln\frac R{s_\de} }
g(y,z_i)\right) -\kappa_j-\frac{2\pi q}{|\ln\ep|}\right)_+^{p}.
\end{split}
\end{equation}

Denote $Z=(z_1,\cdots,z_m)\in \mathbb{R}^{2m}$. In this paper, we
always assume that $z_j\in\Om$ satisfies

\begin{equation}\label{2.5}
d(z_j,\partial\Om)\ge \varrho>0,\quad |z_i-z_j|\ge \varrho^{\bar
L},\quad i, j=1,\cdots,m,\; i\ne j,
\end{equation}
where $\varrho>0$ is a fixed small constant and $\bar L>0$ is a
fixed large constant.

\begin{lemma}\label{l2.1}
For $\de>0$ small,  there exist $(s_{\de,1}(Z),\cdots,s_{\de,m}(Z))$
and $(a_{\de,1}(Z),\cdots,a_{\de,m}(Z))$ satisfying the following
system
\begin{equation}\label{2.6}
\de^{2/(p-1)}s_i^{-2/(p-1)}\phi^\prime(1)=\frac{a_{i}}{\ln(s_i/R)},\qquad
i=1,\cdots,m
\end{equation}
and
\begin{equation}\label{2.7}
a_{i}= \kappa_i + \frac{2\pi
q(z_i)}{|\ln\ep|}+\frac{g(z_i,z_i)}{\ln\frac R {s_{i}} } a_{i}
-\sum_{j\ne i}^m \frac{\bar G(z_i,z_j)}{\ln\frac R{s_{j}} }a_{j}
,\qquad i=1,\cdots, m,
\end{equation}
where for $i=1,\cdots,m$,
$$
\bar G(x,z_i)=\ln \frac R{|x-z_i|} -g(x,z_i).
$$
\end{lemma}
\begin{proof}
We will show that system \eqref{2.6}-\eqref{2.7} has a solution
$(s_1,\cdots,s_m, a_1,\cdots,a_m)$ in
$\mathcal{D}\doteq[\frac{\de}{|\ln\de|},\de |\ln\de|\,]^m\times
\Pi_{i=1}^m[\frac{\kappa_i}{2},\,\frac{3\kappa_i}{2}]$. It is easy
to see, for fixed $(s_1,\cdots,s_m)$ with $0<\de<\de^*$ small, that
\eqref{2.7} has a solution $(a_1,\cdots,a_m)$ depending on
$(s_1,\cdots,s_m)$, such that $\frac{\kappa_i}{2}\leq
a_i\leq\frac{3\kappa_i}{2}$. For such $(a_1,\cdots,a_m)$ and for
$i=1,\cdots, m$ define
\[
\theta_i(s_1,\cdots,s_m)=\frac{s_i^\frac{2}{p-1}}{\ln\frac{R}{s_i}}+\frac{\phi^\prime(1)}{a_i}\de^\frac{2}{p-1},
\]
then it is easy to verify that
\[
\left\{
\begin{array}{ll}
\theta_i(s_1,\cdots,s_m)>0,& s_i=\de |\ln\de |, s_j\in [\frac{\de}{|\ln\de |},\,\,\de |\ln\de |\,] \,\,{\rm for}\,\, j=1,\cdots,m, j\neq i, \\
\theta_i(s_1,\cdots,s_m)<0,& s_i=\frac{\de}{|\ln\de |},\,\,s_j\in
[\frac{\de}{|\ln\de |},\,\,\de |\ln\de |\,] \,\,\,\,\,{\rm for}\,\,
j=1,\cdots,m, j\neq i.
\end{array}
\right.
\]

By the Poincar\'{e}-Miranda Theorem in \cite{K,M}, we can get
$(s_{\de,1},\cdots,s_{\de,m})$ such that
$\theta_i(s_{\de,1},\cdots,s_{\de, m})=0$. Therefore we have
completed our proof of Lemma ~\ref{l2.1}.

\end{proof}
For simplicity, for given $Z=(z_1,\cdots,z_m)$, in this paper, we
will use $a_{\de,i}$,$s_{\de,i}$ instead of
$a_{\de,i}(Z)$,$s_{\de,i}(Z)$.
\begin{remark}\label{remark2.2}
More precisely, we have the following relation
\begin{equation}\label{r2.2.1}
\frac{1}{\ln\frac{R}{s_{\de,i}}}=\frac{1}{\ln\frac{R}{\ep}}
+O\left(\frac{\ln|\ln\ep|}{|\ln\ep|^2}\right),\, i=1,\cdots,m,
\end{equation}

\begin{equation}\label{r2.2.2}
a_{\de,i}=\kappa_i+\frac{2\pi
q(z_i)}{|\ln\ep|}+\frac{g(z_i,z_i)}{\ln\frac{R}{\ep}}-\sum_{j\neq
i}^m\frac {\bar G(z_i,z_j)}{\ln\frac{R}{\ep}}+
O\Bigl(\frac{\ln|\ln\ep|}{|\ln\ep|^2}\Bigr),\, i=1,\cdots,m,
\end{equation}

\begin{equation}\label{2.10}
\frac{\partial a_{\de,i}}{\partial
z_{j,h}}=O\left(\frac1{|\ln\ep|}\right),\quad~~\frac{\partial
s_{\de,i}}{\partial
z_{j,h}}=O\left(\frac{\ep}{|\ln\ep|}\right),\,i,j=1,\cdots,m,\,h=1,2.
\end{equation}

Indeed, \eqref{r2.2.1} can be deduced from \eqref{2.6}(see
\cite{DY}, for example). \eqref{r2.2.2} can been deduced from
\eqref{r2.2.1} and \eqref{2.7}. Differentiating both sides of
\eqref{2.6} and \eqref{2.7} with respect to $z_{j,\,h}$ we can get a
linear system of $\frac{\partial a_{\de,i}}{\partial z_{j,h}}$ and
$\frac{\partial s_{\de,i}}{\partial z_{j,h}}$, which will deduces
\eqref{2.10}.
\end{remark}

 From now on we will always choose
$(a_{\de,1},\cdots,a_{\de,m})$ and $(s_{\de,1},\cdots,s_{\de,m})$
such that \eqref{2.6} and \eqref{2.7} are satisfied. For
$(a_{\de,1},\cdots,a_{\de,m})$ and $(s_{\de,1},\cdots,s_{\de,m})$
chosen in such a way define
\begin{equation}\label{2.8}
P_{\de, Z,j}=
PW_{\de,z_j,a_{\de,j}},\,\,P_{\de,Z}=\sum_{j=1}^mP_{\de, Z,j}.
\end{equation}

 Then,  we find that for $x\in
B_{L s_{\de,i}}(z_i)$, where $L>0$ is any fixed constant,

\[
\begin{split}
& P_{\de, Z,i}(x)-\kappa_i-\frac{2\pi q(x)}{|\ln\ep|}= W_{\de,z_i,
a_{\de,i}}(x) -\frac {a_{\de,i}}{\ln \frac{R}{s_{\de,i} }}
g(x,z_i)-\kappa_i-\frac{2\pi q(x)}{|\ln\ep|}
\\
=&  W_{\de,z_i, a_{\de,i}}(x)- \kappa_i-\frac {a_{\de,i}}{\ln
\frac{R}{s_{ \de,i}}}
 g(z_i,z_i)-\frac {a_{\de,i}}{\ln \frac{R}{s_{\de,i}}}\Bigl(
 \left\langle D g(z_i,z_i), x-z_i\right\rangle +O( |x-z_i|^2)\Bigr)\\
\quad&-\frac{2\pi q(z_i)}{|\ln\ep|}-\frac{2\pi}{|\ln\ep|}\left(\left\langle Dq(z_i),x-z_i\right\rangle+O(|x-z_i|^2)\right)\\
 =&  W_{\de,z_i, a_{\de,i}}(x)- \kappa_i-\frac{2\pi q(z_i)}{|\ln\ep|}-\frac{2\pi}{|\ln\ep|}\left\langle
Dq(z_i),x-z_i\right\rangle\\
\quad&-\frac {a_{\de,i}}{\ln \frac{R}{s_{ \de,i}}}
 g(z_i,z_i)-\frac {a_{\de,i}}{\ln \frac{R}{s_{\de,i}}}
 \left\langle D g(z_i,z_i), x-z_i\right\rangle +O\left(\frac{s_{\de,i}^2}{|\ln\ep|}\right),
\end{split}
\]
and for $j\ne i$ and  $x\in B_{Ls_{\de,i}}(z_i)$, by \eqref{2.2}

\[
\begin{split}
&  P_{\de, Z,j}(x)=W_{\de,
z_j,a_{\de,j}}(x)-\frac{a_{\de,j}}{\ln\frac R{s_{\de,j}} }
  g(x,z_j)=
 \frac{a_{\de,j}}{\ln\frac R{s_{\de,j}}
}   \bar G(x,z_j)\\
=&
  \frac{a_{\de,j}}{\ln\frac R{s_{\de,j}}
}     \bar  G(z_i,z_j)+\frac{a_{\de,j}}{\ln\frac R{s_{\de,j}} }
\left\langle D  \bar G(z_i,z_j),x-z_i\right\rangle +
 O\Bigl( \frac{s_{\de,i}^2 }{|\ln\ep|
} \Bigr).
\end{split}
\]
So,  by using \eqref{2.7}, we obtain

\begin{equation}\label{2.9}
\begin{split}
&  P_{\de, Z}(x)-\kappa_i-\frac{2\pi q(x)}{|\ln\ep|}\\
=&
 W_{\de,z_i, a_{\de,i}}(x)- a_{\de,i} -\frac{2\pi}{|\ln\ep|}\left\langle Dq(z_i),x-z_i\right\rangle-\frac {a_{\de,i}}{\ln \frac{R}{s_{\de,i}}}
 \left\langle D g(z_i,z_i), x-z_i\right\rangle \\
& +\sum_{j\ne i}^m \frac{a_{\de,j} }{\ln\frac R{s_{\de,j}} }
\left\langle D \bar G(z_i,z_j),x-z_i\right\rangle
  +O\left( \frac{s_{\de,i}^2 }{|\ln\ep| } \right),\quad x\in B_{L
s_{\de,i}}(z_i).
\end{split}
\end{equation}

We end this section by giving the following formula which can be
obtained by direct computation and will be used in the next two
sections.

\begin{equation}\label{2.11}
\begin{array}{ll}
 \displaystyle\frac{\partial W_{\de,z_i,a_{\de,i}}(x)}{\partial z_{i,h}}&\\
 =\left\{
 \begin{array}{ll}
\displaystyle\frac{1}{\de}\Bigl(\frac{a_{\de,i}}{|\phi^\prime(1)||\ln\frac{R}{s_{\de,i}}|}\Bigr)^{(p+1)/2}
\phi^\prime\bigl(\frac{|x-z_i|}{s_{\de,i}}\bigr)
\frac{z_{i,h}-x_h}{|x-z_i|}+ O\left(\frac{1}{|\ln\ep|}\right),
 ~~x\in B_{s_{\de,i}}(z_i),\\
 \,\\
\displaystyle-\frac{a_{\de,i}}{\ln\frac{R}{s_{\de,i}}}\frac{z_{i,h}-x_h}{|x-z_i|^2}+O\left(\frac{1}{|\ln\ep|}\right),
 \qquad\qquad\qquad\qquad\qquad\quad x\in \Omega\setminus B_{s_{\de,i}}(z_i).
\end{array}
\right.\\
\end{array}
\end{equation}

\section{the reduction}

Let

\[
w(x)=
\begin{cases}
\phi(|x|), &|x|\le 1,\\
\phi^\prime(1)\ln |x|,  & |x|>1.
\end{cases}
\]
Then $w\in C^1(\mathbb{R}^2)$. Since $\phi^\prime(1)<0$ and $\ln
|x|$ is harmonic for $|x|>1$, we see that $w$ satisfies

\begin{equation}\label{3.1}
-\Delta w= w_+^{p}, \quad \text{in}\; \mathbb{R}^2.
\end{equation}
Moreover, since $w_+$ is Lip-continuous,  by the Schauder estimate,
$w\in C^{2,\alpha}$ for any $\alpha\in (0,1)$.

Consider the following problem:

\begin{equation}\label{3.2}
-\Delta v-  pw_+^{p-1} v=0,\quad v\in L^\infty(\mathbb{R}^2),
\end{equation}
It is easy to see  that $\frac{\partial w}{\partial x_i}$, $i=1,2,$
is a solution of \eqref{3.2}. Moreover, from Dancer and Yan
\cite{DY}, we know that $w$ is also non-degenerate, in the sense
that the kernel of the operator $Lv:=-\Delta v- pw_+^{p-1} v,~~v\in
D^{1,2}(\mathbb{R}^2)$ is spanned by $\bigl\{\frac{\partial
w}{\partial x_1}, \frac{\partial w}{\partial x_2}\bigr\}$.

Recall that $Z=(z_1,\cdots,z_m)$, and
 $z_j\in\Om$ satisfies

\begin{equation}\label{3.3}
d(z_j,\partial\Om)\ge \varrho>0,\quad |z_i-z_j|\ge \varrho^{\bar
L},\qquad i\ne j,
\end{equation}
where $\varrho>0$ is a fixed small constant,  and $\bar L>0$ is a
large constant.

Let $P_{\de,Z,j}$ be the function defined in \eqref{2.8}.  Set

\[
F_{\de,Z}=\left\{u: u\in L^p(\Om),\; \int_\Om
 \frac{\partial
P_{\de,Z,j}}{\partial z_{j,h}}u =0, j=1,\cdots,m,\;h=1,2 \right\},
\]
and

\[
E_{\de,Z}=\left\{u:\;u\in W^{2,p}(\Om)\cap H_0^1(\Om), \int_\Om
 \Delta\left( \frac{\partial
P_{\de,Z,j}}{\partial z_{j,h}}\right)u =0, j=1,\cdots,m,\; h=1,2
\right\}.
\]

For any $u\in L^p(\Om)$, define $Q_\de u$  as follows:
\[
Q_\de u= u-\sum_{j=1}^m \sum_{h=1}^2
b_{j,h}\left(-\de^2\Delta\Bigl(\frac{\partial P_{\de,Z,j}}{\partial
z_{j,h}}\Bigr)
 \right),
\]
where  the constants $b_{j,h}$, $j=1,\cdots,m$, $h=1, 2$, satisfy

\begin{equation}\label{3.4}
\sum_{j=1}^m \sum_{h=1}^2 b_{j,h}\left(-\de^2\int_{\Om}\Delta\Bigl(
\frac{\partial P_{\de,Z,j}}{\partial z_{j, h}}\Bigr) \frac{\partial
P_{\de,Z,i}}{\partial z_{i,\bar h}}\right)=\int_{\Om} u
\frac{\partial P_{\de,Z,i}}{\partial z_{i,\bar h}}.
\end{equation}

Since  $\int_\Om
 \frac{\partial
P_{\de,Z,j}}{\partial z_{j,h}} Q_\de u=0$, the operator $Q_\de$ can
be regarded as a projection from $L^p(\Omega)$ to $F_{\de,Z}$. In
order to show that we can solve \eqref{3.4} to obtain $b_{j,h}$, we
just need the following estimate ( by \eqref{2.10} and
\eqref{2.11}):

\begin{equation}\label{3.5}
\begin{split}
&-\de^2\int_{\Om}\Delta\Bigl( \frac{\partial P_{\de,Z,j}}{\partial
z_{j, h}}\Bigr)
\frac{\partial P_{\de,Z,i}}{\partial z_{i,\bar h}}\\
=&
p\int_{\Om}\bigl(W_{\de,z_j,a_{\de,j}}-a_{\de,j}\bigr)_+^{p-1}\left(\frac{\partial
W_{\de,z_{j},a_{\de,j}}}{\partial z_{j, h}}-\frac{\partial
a_{\de,j}}{\partial z_{j,h}}\right) \frac{\partial
P_{\de,Z,i}}{\partial z_{i,\bar h}}
\\
=& \delta_{ij h\bar h}\frac
c{|\ln\ep|^{p+1}}+O\left(\frac{\ep}{|\ln\ep|^{p+1}}\right),
\end{split}
\end{equation}
where $c>0$ is a constant, $\delta_{ijh \bar h}=1$,  if $i=j$ and
$h=\bar h$;  otherwise, $\delta_{ijh \bar h}=0$.

 Set

\[
L_\de u= -\de^2 \Delta u-\sum_{j=1}^m p\chi_{\Om_j}\left(
P_{\de,Z}-\kappa_j-\frac{2\pi q(x)}{|\ln\ep|}\right)_+^{p-1}u.
\]

  We have the following lemma.

\begin{lemma}\label{l31}
 There are constants $\rho_0>0$ and $\de_0>0$, such that for any
$\de\in (0,\de_0]$,
 $Z$ satisfying \eqref{3.3},  $ u\in E_{\de,Z}$ with
$Q_\de L_\de u =0$ in $\Omega\setminus \cup_{j=1}^m
 B_{L s_{\de,j}}
(z_j)$ for some $L>0$ large,  then

\[
\|Q_\de  L_\de u\|_{L^p(\Om)}  \ge
\frac{\rho_0\de^{\frac{2}{p}}}{|\ln\de|^{\frac{(p-1)^2}{p}}}
\|u\|_{L^\infty(\Om)}.
\]

\end{lemma}

\begin{proof}
 Set $s_{n,j}=s_{\de_n,j}$.
We will use $\|\cdot\|_p, \|\cdot\|_\infty$ to denote
$\|\cdot\|_{L^p(\Om)}$ and  $ \|\cdot\|_{L^\infty(\Om)}$
respectively.

We argue by contradiction. Suppose that there are $\de_n\to 0$,
$Z_n$ satisfying \eqref{3.3} and $u_n\in E_{\de_n,Z_n}$ with
$Q_{\de_n}L_{\de_n} u_n =0$ in $\Omega\setminus \cup_{j=1}^m
 B_{L s_{n,j}}
(z_{j,n})$, $\|u_n\|_\infty =1$, such that
\[
\|Q_{\de_n} L_{\de_n } u_n\|_{p} \le
\frac{1}{n}\frac{\de_n^{\frac{2}{p}}}{|\ln\de_n|^{\frac{(p-1)^2}{p}}}.
\]

Firstly, we estimate $b_{j,h,n}$ in the following formula:

\begin{equation}\label{3.6}
Q_{\de_n} L_{\de_n } u_n= L_{\de_n } u_n-\sum_{j=1}^m \sum_{h=1}^2
b_{j,h,n} \left(-\de_n^2\Delta\frac{\partial
P_{\de_n,Z_n,j}}{\partial z_{j,h}}\right).
\end{equation}

For each fixed $i$, multiplying \eqref{3.6} by $
 \frac{\partial
P_{\de_n,Z_n,i}}{\partial z_{i,\bar h}}$, noting that

\[
 \int_\Om\bigl( Q_{\de_n} L_{\de_n } u_n\bigr)
 \frac{\partial
P_{\de_n,Z_n,i}}{\partial z_{i,\bar h}}=0,
\]
we obtain

\[
\begin{split}
& \int_\Om
  u_n\, L_{\de_n} \left(\frac{\partial
P_{\de_n,Z_n,i}}{\partial z_{i,\bar h}} \right)= \int_\Om\bigl(
  L_{\de_n} u_n\bigr)\, \frac{\partial
P_{\de_n,Z_n,i}}{\partial z_{i,\bar h}} \\
&=\sum_{j=1}^m
 \sum_{\bar h=1}^2 b_{j,h,n} \int_\Om\left(-\de_n^2\Delta\frac{\partial P_{\de_n,Z_n,j}}{\partial
 z_{j,h}}\right)
\frac{\partial P_{\de_n,Z_n,i}}{\partial z_{i,\bar h}}\\
\end{split}
\]
Using \eqref{2.9} and
 Lemma~\ref{al1}, we obtain

\[
\begin{split}
&\int_\Om
  u_n\, L_{\de_n} \left(\frac{\partial
P_{\de_n,Z_n,i}}{\partial z_{i,\bar h}}\right)\\
&=\int_\Om\left(-\de_n^2\Delta \left(\frac{\partial
P_{\de_n,Z_n,i}}{\partial z_{i,\bar
h}}\right)-\sum_{j=1}^mp\chi_{\Om_j}\left(
P_{\de_n,Z_n}-\kappa_j-\frac{2\pi
q(x)}{|\ln\ep_n|}\right)_+^{p-1}\frac{\partial
P_{\de_n,Z_n,i}}{\partial z_{i,\bar h}}\right)u_n\\
&=p\int_\Om\left(W_{\de_n,z_{i,n},a_{\de_n,i}}-a_{\de_n,i}\right)_+^{p-1}\left(\frac{\partial
W_{\de_n,z_{i,n},a_{\de_n,i}}}{\partial z_{i,\bar h}}-\frac{\partial
a_{\de_n,i}}{\partial z_{i,\bar h}}\right)u_n\\
&\quad-\sum_{j=1}^mp\int_{\Omega_j}\left(W_{\de_n,z_{j,n},a_{\de_n,j}}-a_{\de_n,j}+O\left(\frac{s_{n,j}}{|\ln\ep_n|}\right)\right)_+^{p-1}\frac{\partial
P_{\de_n,Z_n,i}}{\partial
z_{i,\bar h}}u_n\\
&=O\left(\frac{\ep_n^2}{|\ln\ep_n|^p}\right).
\end{split}
\]

Using \eqref{3.5}, we find that
$$
b_{i,h,n}=O\left(\ep_n^2|\ln\ep_n|\right).
$$
Therefore,
\[
\begin{split}
&\sum_{j=1}^m\sum_{h=1}^2b_{j,h,n}\left(-\de_n^2\Delta\frac{\partial
P_{\de_n,Z_n,j}}{\partial z_{j,h}}\right)\\
&=p\sum_{j=1}^m\sum_{h=1}^2b_{j,h,n}\left(W_{\de_n,z_{j,n},a_{\de_n,j}}-a_{\de_n,j}\right)_+^{p-1}
\left(\frac{\partial W_{\de_n,z_{j,n},a_{\de_n,j}}}{\partial
z_{j,h}}-\frac{\partial a_{\de_n,j}}{\partial z_{j,h}}\right)\\
&=O\left(\sum_{j=1}^m\sum_{h=1}^2  \frac{\ep_n^{\frac{2}{p}-1}|b_{j,h,n}|}{|\ln\ep_n|^{p}}\right)\\
&=O\left(\frac{\ep_n^{\frac{2}{p}+1}}{|\ln\ep_n|^{p-1}}\right)\quad
\text{in}~~L^p(\Om).
\end{split}
\]

Thus, we obtain

\[
L_{\de_n}u_n = Q_{\de_n}L_{\de_n}u_n
+O\left(\frac{\ep_n^{\frac{2}{p}+1}}{|\ln\ep_n|^{p-1}}\right)
 =O\left(\frac 1n\frac{\de_n^{\frac{2}{p}}}{|\ln\de_n|^{\frac{(p-1)^2}{p}}}\right).
\]

For any fixed $i$, define

\[
\tilde u_{i,n} (y)= u_n(s_{n,i} y+z_{i,n}).
\]

Let
\[
\tilde L_n u= -\Delta u
-\sum_{l=1}^mp\frac{s_{n,i}^2}{\de_n^2}\chi_{\Om_l}\left(
P_{\de_n,Z_n}(s_{n,i}y+z_{i,n})-\kappa_l-\frac{2\pi
q}{|\ln\ep_n|}\right)_+^{p-1}u,
\]
Then

\[
s_{n,i}^{\frac{2}{p}}\times\frac{\de_n^2}{s_{n,i}^{2}}\|\tilde L_n
\tilde u_{i,n}\|_p= \| L_{\de_n} u_n\|_p.
\]

Noting that
$$ \left(\frac{\de_n}{s_{n,i}}\right)^2=O\left(\frac{1}{|\ln\de_n|^{p-1}}\right),$$
we find that
$$
L_{\de_n}u_n=o\left(\frac{\de_n^{\frac{2}{p}}}{|\ln\de_n|^{\frac{(p-1)^2}{p}}}\right).
$$
 As a result,

\[
\tilde  L_{n } \tilde u_{i,n}
 =o(1),\quad \text{in}\; L^p(\Omega_n),
\]
where $\Omega_n=\bigl\{y: s_{n,i} y+z_{i,n}\in\Omega\bigr\}$.

Since $\|\tilde u_{i,n}\|_\infty=1$, by the regularity theory of
elliptic equations,  we may assume that

\[
\tilde u_{i,n}\to u_i,  \quad \text{in}\; C_{loc}^1(\mathbb{R}^2).
\]

It is easy to see that
\[
\begin{split}
&\sum_{l=1}^m\frac{s_{n,i}^2}{\de_n^2}\chi_{\Om_l}\left(
P_{\de_n.Z_n}(s_{n,i}y+z_{i,n})-\kappa_l-\frac{2\pi q}{|\ln\ep_n|}\right)_+^{p-1}\\
&=\frac{s_{n,i}^2}{\de_n^2}\left(W_{\de_n,z_{i,n},a_{\de_n,i}}-a_{\de_n,i}
+O\left(\frac{s_{n,i}}{|\ln\ep_n|}\right)\right)_+^{p-1}+o(1)\\
&\rightarrow w_+^{p-1}.
\end{split}
 \]
  Then, by Lemma~\ref{al1}, we find that
  $u_i$ satisfies

\[
-\Delta u_i-pw_+^{p-1} u_i= 0.
\]
Now from the Proposition 3.1 in \cite{DY}, we have

\begin{equation}\label{3.7}
u_i= c_1 \frac{\partial w}{\partial x_1}+ c_2 \frac{\partial
w}{\partial x_2}.
\end{equation}

Since

\[
\int_\Om \Delta \bigl(\frac{\partial P_{\de_n,Z_n,i}}{\partial z_{i,
h}}\bigr) u_n =0,
\]
we find that

\[
\int_{\mathbb{R}^2}\phi_+^{p-1} \frac{\partial \phi}{\partial z_h}
u_i =0,
\]
which, together with \eqref{3.7}, gives $u_i=0$. Thus,

\[
\tilde u_{i,n} \to 0,\quad \text{in}\; C^1(B_{L}(0)),
\]
for any $L>0$, which implies that $u_n=o(1)$ on $\partial
B_{Ls_{n.i}}(z_{i,n})$.

By assumption,

\[
Q_{\de_n} L_{\de_n} u_n = 0,\quad\text{in}\; \Om\setminus
\cup_{i=1}^k B_{L s_{n,i}}(z_{i,n}).
\]

On the other hand, by Lemma~\ref{al1}, for $j=1,\cdots,m$, we have

\[
\left(P_{\de_n,Z_n} -\kappa_j-\frac{2\pi q(x)}{|\ln\ep_n|} \right)_+
=0, \quad x\in \Om_j\setminus  B_{L s_{n,j}}(z_{j,n}).
\]
Thus, we find that

\[
-\Delta u_n  =0,\quad \text{in}~\Om\setminus  \cup_{i=1}^m  B_{L
s_{n,i}}(z_{i,n}).
\]
However, $u_n=0$ on $\partial\Om$ and $u_n=o(1)$ on $\partial
B_{Ls_{n,i}}(z_{i,n})$, $i=1,\cdots,m$. So we have

\[
u_n=o(1).
\]
 This is a contradiction.

\end{proof}

\begin{proposition}\label{p32}

$Q_\de L_\de$ is one to one and onto from $E_{\de,Z}$ to
$F_{\de,Z}$.

\end{proposition}

\begin{proof}

Suppose that $Q_\de L_\de u=0$. Then, by  Lemma~\ref{l31},  $u=0$.
Thus,  $Q_\de L_\de$ is one to one.

Next,  we prove that  $Q_\de L_\de $ is an onto map from $E_{\de,Z}$
to $F_{\de,Z}$.

Denote

\[
\tilde E= \Bigl\{ u:  u\in H_0^1(\Omega), \; \int_{\Omega}
D\frac{\partial P_{\de,Z,j}}{\partial z_{j,h} } Du=0,\;
j=1,\cdots,m, h=1,  2\Bigr\}.
\]
Note that $E_{\de,Z}=\tilde E\cap W^{2,p}(\Omega)$.

 For any
$\tilde h\in F_{\de,Z}$, by the Riesz representation theorem, there
is a unique $u\in H_0^1(\Om)$, such that

\begin{equation}\label{3.8}
\de^2\int_\Om Du D\varphi =\int_\Om \tilde h \varphi,\quad \forall\;
\varphi\in   H_0^1(\Om).
\end{equation}
On the other hand,  from $\tilde h\in F_{\de,Z}$, we find that $u\in
\tilde E$. Moreover, by the  $L^p$-estimate, we deduce that $u\in
W^{2,p}(\Omega)$. As a result, $u\in E_{\de,Z}$. Thus, we see that
$Q_\de (-\de^2\Delta )=-\de^2\Delta$ is an one to one and onto map
from $ E_{\de,Z}$ to $F_{\de,Z}$. On the other hand, $Q_\de L_\de
u=h$ is equivalent to

\begin{equation}\label{3.9}
u=p\de^{-2} (-Q_\de\Delta )^{-1}\left[
Q_\de\left(\sum_{j=1}^m\chi_{\Om_j} \left(
P_{\de,Z}-\kappa_j-\frac{2\pi q(x)}{|\ln\ep|}\right)_+^{p-1}
u\right)\right]+\de^{-2} (-Q_\de\Delta )^{-1} h,\quad u\in
E_{\de,Z}.
\end{equation}
It is easy to check that $\de^{-2} (-Q_\de\Delta )^{-1}\left[
Q_\de\left(\sum_{j=1}^m\chi_{\Om_j} \left(
P_{\de,Z}-\kappa_j-\frac{2\pi q(x)}{|\ln\ep|}\right)_+^{p-1}
u\right)\right]$ is a compact operator in $E_{\de,Z}$. By the
Fredholm alternative, \eqref{3.9} is solvable if and only if

\[
u=p\de^{-2} (-Q_\de\Delta )^{-1}\left[
Q_\de\left(\sum_{j=1}^m\chi_{\Om_j} \left(
P_{\de,Z}-\kappa_j-\frac{2\pi q(x)}{|\ln\ep|}\right)_+^{p-1}
u\right)\right]
\]
has trivial solution, which is true since
 $Q_\de L_\de$ is a one to one map.
Thus the result follows.

\end{proof}

Now  consider the equation

\begin{equation}\label{3.10}
Q_\de L_\de \omega=
 Q_\de l_\de + Q_\de R_\de(\omega),
\end{equation}
where

\begin{equation}\label{3.11}
l_\de =\sum_{j=1}^m \chi_{\Om_j}\left( P_{\de,Z}
-\kappa_j-\frac{2\pi q(x)}{|\ln\ep|} \right)_+^{p}-\sum_{j=1}^m
\left(W_{\de,z_j,a_{\de,j}}-a_{\de,j}\right)_+^{p},
\end{equation}
and

\begin{equation}\label{3.12}
\begin{split}
R_\de(\omega)=& \sum_{j=1}^m \chi_{\Om_j}\left(P_{\de,Z}
-\kappa_j+\omega-\frac{2\pi q(x)}{|\ln\ep|} \right)_+^{p}
-\sum_{j=1}^m \chi_{\Om_j}
 \left(P_{\de,Z}
-\kappa_j-\frac{2\pi q(x)}{|\ln\ep|}\right)_+^{p} \\
&-\sum_{j=1}^m \chi_{\Om_j}p\left(P_{\de,Z} -\kappa_j-\frac{2\pi
q(x)}{|\ln\ep|}\right)_+^{p-1}\omega.
\end{split}
\end{equation}

Using Proposition~\ref{p32}, we can rewrite \eqref{3.10} as

\begin{equation}\label{3.13}
\omega =G_\de\omega =: (Q_\de L_\de)^{-1} Q_\de \bigl(
  l_\de +   R_\de(\omega)\bigr).
\end{equation}

The next Proposition enables us to reduce the problem of finding a
solution for \eqref{1} to a finite dimensional problem.

\begin{proposition}\label{p33}

There is an $\de_0>0$, such that for any $\de\in (0,\de_0]$ and  $Z$
 satisfying  \eqref{3.3},  \eqref{3.10} has a unique solution $\omega_\de\in
 E_{\de,Z}$, with

\[
\|\omega_\de\|_\infty =O\Bigl(\de|\ln\de|^{\frac{p-1}{2}}\Bigr).
\]
\end{proposition}

\begin{proof}

It follows from Lemma~\ref{al1} that if $L$ is large enough, $\de$
is small then

\[
 \left(P_{\de,Z}
-\kappa_j-\frac{2\pi q(x)}{|\ln\ep|} \right)_+=0, \quad x\in
\Om_j\setminus  B_{Ls_{\de,j}}(z_j),j=1,\cdots,m.
\]

Let

\[
M=  E_{\de,Z}\cap\Bigl\{ \|\omega\|_\infty\le
\de|\ln\de|^{\frac{p-1}{2}}\Big\}.
\]
Then $M$ is complete under $L^\infty$ norm  and $G_\de$ is a map
from $ E_{\de,Z}$ to $ E_{\de,Z}$. We will show that $G_\de $ is a
contraction map from $M$ to $M$.

Step~1.  $G_\de$ is a map from $M$ to $M$.

 For any $\omega\in M$, similar to Lemma~\ref{al1},
it is easy to  prove that for large $L>0$, $\de$ small

\begin{equation}\label{3.14}
 \left(P_{\de,Z}+\omega
-\kappa_j-\frac{2\pi q(x)}{|\ln\ep|}\right)_+=0, \quad \text{in}\;
\Om_j\setminus  B_{Ls_{\de,j}}(z_j).
\end{equation}
Note also that for any $u\in L^\infty(\Om)$,

\[
Q_\de u= u\quad \text{in}\; \Om\setminus  \cup_{j=1}^m
B_{Ls_{\de,j}}(z_j).
\]
Therefore, using Lemma~\ref{al1}, \eqref{3.11} and \eqref{3.12}, we
find that for any $\omega\in M$,

\[
 Q_\de l_\de + Q_\de R_\de(\omega)= l_\de + R_\de(\omega)
=0, \quad \text{in}\; \Om\setminus  \cup_{j=1}^m
B_{Ls_{\de,j}}(z_j).
\]
So, we can apply Lemma~\ref{l31} to obtain

\[
\| (Q_\de L_\de)^{-1} \bigl(
 Q_\de l_\de + Q_\de R_\de(\omega)\bigr)\|_\infty
\le \frac{C|\ln\de|^{\frac{(p-1)^2}{p}}}{\de^{\frac{2}{p}}} \| Q_\de
l_\de  +  Q_\de R_\de(\omega)\|_p.
\]

Thus, for any
 $\omega\in M$, we have

\begin{equation}\label{3.15}
\begin{split}
\| G_\de(\omega)\|_\infty =& \| (Q_\de L_\de)^{-1} Q_\de \bigl(
  l_\de +  R_\de(\omega)\bigr)\|_\infty\\
\le & \frac{C|\ln\de|^{\frac{(p-1)^2}{p}}}{\de^{\frac{2}{p}}}\|Q_\de
\bigl( l_\de+ R_\de(\omega)\bigr)\|_p.
\end{split}
\end{equation}

It follows from \eqref{3.4}--\eqref{3.5} that the constant
$b_{j,h}$, corresponding to $u\in L^\infty(\Om)$,  satisfies

\[
|b_{j,h}| \le C |\ln\de|^{p+1}\sum_{i,\,\bar h}
 \int_\Om \Bigl|\frac{\partial P_{\de,Z,i}}
{\partial z_{i,\bar h}}\Bigr| |u|.
\]

 Since

\[
l_\de+ R_\de(\omega)=0, \quad \text{in}\; \Om\setminus  \cup_{j=1}^m
B_{Ls_{\de,j}}(z_j),
\]
 we find that the constant
$b_{j,h}$, corresponding to $l_\de+ R_\de(\omega)$ satisfies

\[
\begin{split}
|b_{j,h}|\le & C|\ln\de|^{p+1} \sum_{i,\,\bar h} \left(\sum_{j=1}^m
\int_{B_{Ls_{\de,j} }(z_j)}
 \Bigl|\frac{\partial P_{\de,Z,i}}
{\partial z_{i,\bar h}}\Bigr||l_\de+ R_\de(\omega)
|\right)\\
\le & C\ep^{1-\frac{2}{p}}|\ln\ep|^{p}\|l_\de+ R_\de(\omega) \|_p.
\end{split}
\]
As a result,

\[
\begin{split}
&\|Q_\de (l_\de+ R_\de(\omega))\|_p \\
\le & \| l_\de+ R_\de(\omega)\|_p+C \sum_{j,\,h} |b_{j,h}| \left\|
-\de^2\Delta\Bigl(\frac{\partial P_{\de,Z,j}}{\partial
z_{j,h}}\Bigr)
\right\|_p\\
\le & C \| l_\de\|_p+ C\| R_\de(\omega) \|_p.
\end{split}
\]

On the other hand, from Lemma~\ref{al1} and \eqref{2.9}, we can
deduce

\[
\begin{split}
\|l_\de\|_p =&\left\|\sum_{j=1}^m\chi_{\Om_j}\left(
P_{\de,Z}-\kappa_j-\frac{2\pi q(x)}{|\ln\ep|}\right)_+^{p}-\sum_{j=1}^m\bigl(W_{\de,z_j,a_{\de,j}}-a_{\de,j}\bigr)_+^{p}\right\|_p\\
\le&\sum_{j=1}^m\frac{Cs_{\de,j}}{|\ln\ep|}\Big\|\bigl(W_{\de,z_j,a_{\de,j}}-a_{\de,j}\bigr)_+^{p-1}\Big\|_p\\
=&O\left(\frac{\de^{1+\frac{2}{p}}}{|\ln\de|^{\frac
{p-1}2+\frac{1}p}}\right).
\end{split}
\]

For the estimate of $\| R_\de(\omega) \|_p$, we have
\begin{equation}\label{3.16}
\begin{split}
\| R_\de(\omega) \|_p=&\bigg\|\sum_{j=1}^m\chi_{\Om_j}\Big(
P_{\de,Z}-\kappa_j+\omega-\frac{2\pi
q(x)}{|\ln\ep|}\Big)_+^{p}-\sum_{j=1}^m\chi_{\Om_j}\Big(
P_{\de,Z}-\kappa_j-\frac{2\pi q(x)}{|\ln\ep|}\Big)_+^{p}\\
&-\sum_{j=1}^mp\chi_{\Om_j}\Big(
P_{\de,Z}-\kappa_j-\frac{2\pi q(x)}{|\ln\ep|}\Big)_+^{p-1}\omega\bigg\|_p\\
\le&\sum_{j=1}^mC\|\omega\|_\infty^2\left\|\chi_{\Om_j}\left(
P_{\de,Z}-\kappa_j-\frac{2\pi q(x)}{|\ln\ep|}\right)_+^{p-2}\right\|_p\\
=&O\left(\frac{\de^{\frac{2}{p}}\|\omega\|_\infty^2}{|\ln\de|^{p-3+\frac{1}p}}\right).
\end{split}
\end{equation}

Thus, we obtain

\begin{equation}\label{3.17}
\begin{split}
\| G_\de(\omega)\|_\infty \le &
\frac{C|\ln\de|^{\frac{(p-1)^2}{p}}}{\de^{\frac{2}{p}}}\Bigl(\|
l_\de \|_p+\| R_\de(\omega)\|_p
\Bigr)\\
\le  & C|\ln\de|^{\frac{(p-1)^2}{p}}\left(\frac{\de}{|\ln\de|^{\frac
{p-1}2+\frac{1}p}}+\frac{\|\omega\|_\infty^2}{|\ln\de|^{p-3+\frac{1}p}}\right)\\
\le&\de|\ln\de|^{\frac{p-1}2}
\end{split}
\end{equation}

Thus, $G_\de$ is a map from $M$ to $M$.

 Step~2.  $G_\de$ is a
contraction map.

In fact, for any $\omega_i\in M$, $i=1,2 $, we have

\[
G_\de \omega_1-G_\de \omega_2= (Q_\de L_\de)^{-1} Q_\de \bigl(
R_\de(\omega_1)-R_\de(\omega_2)\bigr).
\]
Noting that

\[
R_\de(\omega_1)=R_\de(\omega_2)=0,\quad\text{in}\; \Om\setminus
\cup_{j=1}^m B_{Ls_{\de,j}}(z_j),
\]
we can deduce as in Step~1 that
\[
\begin{split}
\|G_\de \omega_1-G_\de \omega_2\|_\infty\le& \frac{C|\ln\de|^{\frac{(p-1)^2}{p}}}{\de^{\frac{2}{p}}}\|R_\de(\omega_1)-R_\de(\omega_2)\|_p\\
\le&C|\ln\de|^{p-1}\left(\frac{\|\omega_1\|_\infty}{|\ln\de|^{p-2}}+\frac{\|\omega_2\|_\infty}{|\ln\de|^{p-2}}\right)\|\omega_1-\omega_2\|_\infty\\
\le&C\de|\ln\de|^{\frac {p+1}2}\|\omega_1-\omega_2\|_\infty \le\frac
12 \|\omega_1-\omega_2\|_\infty.
\end{split}
\]

Combining Step~1 and Step~2,  we have proved that $G_\de$ is a
contraction map from $M$ to $M$. By
  the
 contraction mapping theorem, there is an unique $\omega_\de\in M$, such
that $\omega_\de= G_\de\omega_\de$. Moreover, it follows from
\eqref{3.17} that

\[
\|\omega_\de\|_\infty\le \de|\ln\de|^{\frac {p-1}2}.
\]

\end{proof}

\section{Proof of The main results}

In this section, we will choose $Z$, such that $ \sum_{j=1}^m
P_{\de,Z,j}+\omega_\de$, where $\omega_\de$ is the map obtained in
Proposition~\ref{p33}, is a solution of \eqref{1}.

Define
\[
I(u)=\frac{\de^2}{2}\int_\Omega
|Du|^2-\sum_{j=1}^m\frac{1}{p+1}\int_\Omega\chi_{\Om_j}\left(u-\kappa_j-\frac{2\pi
q(x)}{|\ln\ep|}\right)_+^{p+1}
\]
and
\begin{equation}\label{K}
 K(Z)= I\left( P_{\de,Z} +\omega_\de\right).
\end{equation}
It is well known that  if $Z$ is a critical point of $K(Z)$, then
$\sum_{j=1}^m P_{\de,Z,j} +\omega_\de$ is a solution of \eqref{1}.

In the following,  we will prove that $K(Z)$ has a critical point.

 \begin{lemma}\label{l42}

 We have

 \[
 K(Z)= I\left(\sum_{j=1}^m P_{\de,Z,j}
\right)+O\left(
 \frac{\ep^3}{|\ln\ep|^p}\right).
\]

 \end{lemma}

\begin{proof}

Recall that

\[
P_{\de,Z}=\sum_{j=1}^m P_{\de,Z,j}.
\]
We have

\[
\begin{split}
K(Z)=& I\bigl( P_{\de,Z}\bigr) +\int_\Om \de^2
DP_{\de,Z}D\omega_\de+\frac{\de^2}{2}\int_\Omega |D\omega_\de|^2\\
&-\frac1{p+1} \sum_{j=1}^m\int_\Om \chi_{\Om_j}\Biggl[ \biggl(
P_{\de,Z} +\omega_\de-\kappa_j-\frac{2\pi q(x)}{|\ln\ep|}
\biggr)_+^{p+1}- \biggl( P_{\de,Z}-\kappa_j-\frac{2\pi
q(x)}{|\ln\ep|} \biggr)_+^{p+1}\Biggr].
\end{split}
\]

Using Proposition~\ref{p33} and \eqref{3.14}, we find

\[
\begin{split}
&  \int_{\Om_j} \Biggl[ \biggl( P_{\de,Z}
+\omega_\de-\kappa_j-\frac{2\pi q(x)}{|\ln\ep|} \biggr)_+^{p+1}-
\biggl(
P_{\de,Z}-\kappa_j-\frac{2\pi q(x)}{|\ln\ep|} \biggr)_+^{p+1}\Biggr]\\
= &\int_{ B_{Ls_{\de,j}}(z_j)} \Biggl[ \biggl( P_{\de,Z}
+\omega_\de-\kappa_j-\frac{2\pi q(x)}{|\ln\ep|} \biggr)_+^{p+1}-
\biggl(
P_{\de,Z}-\kappa_j-\frac{2\pi q(x)}{|\ln\ep|} \biggr)_+^{p+1}\Biggr]\\
=&O\left(\frac{s_{\de,j}^2\|\omega_\de\|_\infty}{|\ln\ep|^{p}}\right)
= O\Bigl(
 \frac{\ep^3}{|\ln\ep|^{p}}\Bigr).
\end{split}
\]

On the other hand,

\[
\begin{split}
&\de^2 \int_\Om DP_{\de,Z}D\omega_\de
 =
 \sum_{j=1}^m\int_\Om
\left(W_{\de,z_j,a_{\de,j}}-a_{\de,j}\right)_+^{p} \omega_\de\\
=&\sum_{j=1}^m
 \int_{\cup_{k=1}^m B_{s_{\de,k}}(z_k)}
(W_{\de,z_j,a_{\de,j}}-a_{\de,j})_+^{p} \omega_\de \\
=& O\Bigl(
 \frac{\ep^3}{|\ln\ep|^{p}}\Bigr).
\end{split}
\]
Finally, we estimate $\de^2\int_\Om |D\omega_\de|^2$.\\
Note that
\[
\begin{split}
-\de^2\Delta\omega_\de=&\sum_{j=1}^m\chi_{\Om_j}\left(P_{\de,Z}+\omega_\de-\kappa_j-\frac{2\pi
q(x)}{|\ln\ep|}\right)_+^{p}
-\sum_{j=1}^m\left(W_{\de,z_j,a_{\de,j}}-a_{\de,j}\right)_+^{p}\\
&+\sum_{j=1}^m\sum_{\bar h=1}^2b_{j,\bar
h}\left(-\de^2\Delta\frac{\partial P_{\de,Z,j}}{\partial z_{j,\bar
h}}\right),
\end{split}
\]
Hence, by \eqref{2.9}, we have
\[
\begin{split}
\de^2\int_\Om|D\omega_\de|^2=&\sum_{j=1}^m\int_{\Om_j}\Biggl[\left(P_{\de,Z}+\omega_\de-\kappa_j-\frac{2\pi
q(x)}{|\ln\ep|}\right)_+^{p}
-\left(W_{\de,z_j,a_{\de,j}}-a_{\de,j}\right)_+^{p}\Biggr]\omega_\de\\
&+\sum_{j=1}^m\sum_{\bar h=1}^2b_{j,\bar
h}\int_\Om\left(-\de^2\Delta\frac{\partial P_{\de,Z,j}}{\partial
z_{j,\bar h}}\right)\omega_\de\\
=&p\sum_{j=1}^m\int_{\Om_j}\left(W_{\de,z_j,a_{\de,j}}-a_{\de,j}\right)_+^{p-1}\left(\frac{s_{\de,j}}{|\ln\ep|}+\omega_\de\right)\omega_\de
+O\left(\sum_{j=1}^m\sum_{\bar h=1}^2\frac{\ep|b_{j,\bar
h}|\|\omega_\de\|_\infty}{|\ln\ep|^{p}}\right)\\
=&O\left(\frac{\ep^4}{|\ln\ep|^{p-1}}\right).
\end{split}
\]

 So we can obtain  that

\[
 K(Z)= I\left(\sum_{j=1}^m P_{\de,Z,j}
\right)+O\left(
 \frac{\ep^3}{|\ln\ep|^{p}}\right).
\]

\end{proof}

\begin{lemma}\label{l43}

 We have

\[
\frac{\partial K(Z)}{\partial z_{i,h}}=\frac{\partial }{\partial
z_{i,h}}I\left(\sum_{j=1}^m P_{\de,Z,j} \right)+O\Bigl(
\frac{\ep^3}{|\ln\ep|^{p-1}}\Bigr).
\]

 \end{lemma}

\begin{proof}
First, we have

\begin{equation}\label{4.4}
\begin{split}
&\frac{\partial K(Z)}{\partial z_{i,h}}=\left\langle I^\prime\Bigl(
P_{\de,Z} +\omega_\de\Bigr),\frac{\partial P_{\de,Z}}{\partial
z_{i,h}}+ \frac{\partial
\omega_\de}{\partial z_{i,h}}\right\rangle\\
=&\frac{\partial }{\partial z_{i,h}}I\Bigl(P_{\de,Z} \Bigr)
+\left\langle I^\prime\big( P_{\de,Z}+\omega_\de\big),
\frac{\partial\omega_\de}{\partial z_{i,h}}\right\rangle
\\
&-\sum_{j=1}^m \int_{\Omega_j}\Biggl[
\biggl(P_{\de,Z}+\omega_\de-\kappa_j-\frac{2\pi
q(x)}{|\ln\ep|}\biggr)_+^{p}-\biggl(P_{\de,Z}-\kappa_j-\frac{2\pi
q(x)}{|\ln\ep|}\biggr)_+^{p}\Biggr]\frac{\partial
P_{\de,Z}}{\partial z_{i,h}}.
\end{split}
\end{equation}

Since $\omega_\de\in E_{\de,Z}$, we have
\[
\int_\Om\left(W_{\de,z_j,a_{\de,j}}-a_{\de,j}\right)_+^{p-1}\left(\frac{\partial
W_{\de,z_j,a_{\de,j}}}{\partial z_{j,h}}-\frac{\partial
a_{\de,j}}{\partial z_{j,h}}\right)\omega_\de=0.
\]
Differentiating the above relation with respect to $z_{i, h}$, we
can deduce

\begin{equation*}
\begin{split}
\Bigg\langle & I^\prime\bigl(P_{\de,Z} +\omega_\de\bigr),
\frac{\partial \omega_\de}{\partial z_{i, h}}
\Bigg\rangle=\sum_{j=1}^m\sum_{\bar h=1}^2 b_{j,\bar h}
\int_\Om\left(-\de^2\Delta\frac{\partial P_{\de,Z,j}}{\partial
z_{j,\bar h}}\right) \frac{\partial \omega_\de}{\partial z_{i, h}}\\
=&\sum_{j=1}^m\sum_{\bar h=1}^2 p b_{j,\bar h}
\int_\Om\left(W_{\de,z_j,a_{\de,j}}-a_{\de,j}\right)_+^{p-1}\left(\frac{\partial
W_{\de,z_j,a_{\de,j}}}{\partial z_{j,\bar h}}-\frac{\partial
a_{\de,j}}{\partial z_{j,\bar h}}\right) \frac{\partial
\omega_\de}{\partial z_{i, h}}\\
=&O\left(\sum_{j=1}^m\sum_{\bar h=1}^2 \frac{\ep|b_{j,\bar
h}|}{|\ln\ep|^{p}}\right)=O\left(\frac{\ep^3}{|\ln\ep|^{p-1}}\right).
\end{split}
\end{equation*}
On the other hand, using \eqref{3.16} (for the definition of
$R_\de(\omega)$, see \eqref{3.12}), we obtain

\[
\begin{split}
&\sum_{j=1}^m\int_{\Omega_j}\Biggl[
\biggl(P_{\de,Z}+\omega_\de-\kappa_j-\frac{2\pi
q(x)}{|\ln\ep|}\biggr)_+^{p}-\biggl(P_{\de,Z}-\kappa_j-\frac{2\pi
q(x)}{|\ln\ep|}\biggr)_+^{p}\Biggr]\frac{\partial
P_{\de,Z,i}}{\partial z_{i,h}}\\
= & \sum_{j=1}^m\int_{\Omega_j}\Biggl[
\biggl(P_{\de,Z}+\omega_\de-\kappa_j-\frac{2\pi q(x)}{|\ln\ep|}\biggr)_+^{p}-\biggl(P_{\de,Z}-\kappa_j-\frac{2\pi q(x)}{|\ln\ep|}\biggr)_+^{p}\\
&-p\biggl(P_{\de,Z}-\kappa_j-\frac{2\pi
q(x)}{\kappa|\ln\ep|}\biggr)_+^{p-1}\omega_\de\Biggr]\frac{\partial
P_{\de,Z,i}}{\partial z_{i,h}} +\sum_{j=1}^m
p\int_{\Omega_j}\Biggl[\biggl(P_{\de,Z}-\kappa_j-\frac{2\pi q(x)}{|\ln\ep|}\biggr)_+^{p-1}\\
&-\bigl(W_{\de,z_j,a_{\de,j}}-a_{\de,j}\bigr)_+^{p-1}\Biggr]\frac{\partial
P_{\de,Z,i}}{\partial z_{i,h}}\omega_\de+O\left(\frac{s_{\de,j}^2\|\omega_\de\|_\infty}{|\ln\ep|^{p}}\right)\\
=&\int_\Omega R_\de(\omega_\de)\frac{\partial P_{\de,Z,i}}{\partial
z_{i,h}} +
\sum_{j=1}^mp\int_{\Omega_j}\Biggl[\biggl(P_{\de,Z}-\kappa_j-\frac{2\pi
q(x)}{|\ln\ep|}\biggr)_+^{p-1}-\bigl(W_{\de,z_i,a_{\de,i}}-a_{\de,i}\bigr)_+^{p-1}\Biggr]\frac{\partial
P_{\de,Z,i}}{\partial
z_{i,h}}\omega_\de\\
&+O\left(\frac{\ep^3}{|\ln\ep|^{p}}\right)\\
 =& O\left(\frac{\ep^3}{|\ln\ep|^{p-1}}\right).
\end{split}
\]
Thus, the estimate follows.
\end{proof}

Define

\[
c_{\de,1}=\frac{ C \de^2
 }{\ln \frac{R}{\ep}
 }-\frac{ \pi\de^2\ln\frac1{\varrho}
 }{|\ln\frac{R}{\ep}|^2
 },\quad
c_{\de,2}=\frac{(C+\eta) \de^2
 }{\ln\frac{R}{\ep}
 },
\]
where $\eta>0$ is a small constant and $\varrho>0$ is a fixed small
constant. Let

\[
D=\bigl\{ Z=(z_1,\cdots,z_m): \; z_i\in\Omega_\varrho,
i=1,\cdots,m,\; |z_i-z_j|\ge \varrho^{\bar L}, i\ne j\bigr\},
\]
where $\Omega_\varrho=\bigl\{y: y\in\Omega,\; d(y,\partial\Omega)\ge
\varrho\bigr\}$,  and $\bar L>0$ is a large constant.

Denote $K^c=\bigl\{Z: Z\in D,\; K(Z)\le c\bigr\}$.  Consider

\[
\left\{
 \begin{array}{ll}
\frac{d Z(t)}{dt}=- DK(Z(t)), \; t\ge 0,&\\ \vspace{0.05cm}\\
Z(0)\in K^{c_{\de,2}}.&
\end{array}
\right.
\]

\begin{lemma}\label{l44}
 $Z(t)$ does not leave $D$ before it reaches $K^{c_{\de,1}}$.
\end{lemma}

\begin{proof}
Note that
\begin{equation}\label{4.5}
h(x,z)=\frac{1}{2\pi}\ln\frac{1}{|x-\bar{z}|}+o(1),\quad
\frac{\partial h(x,z)}{\partial
n}=-\frac{1}{2\pi|x-\bar{z}|}\left\langle\frac{x-\bar{z}}{|x-\bar{z}|},n\right\rangle+o(1),
\end{equation}
if $z$ is close to $\partial\Omega$, where $n$ is the outward normal
unit vector of  $\partial\{x: x\in\Omega, d(x,
\partial\Omega)\leq d(z, \partial\Omega)\}$ and $\bar{z}$ is the
reflection point of $z$ with respect to $\partial\Omega$.

 Suppose that there is $t_0>0$, such that
$Z(t_0)=:(z_1,\cdots,z_m)\in
\partial D$.

\begin{enumerate}
  \item  Suppose that  there are $i, j\in\{1,2,\cdots,m\}$, such that  $i\neq j$ and $|z_i-z_j|=\varrho^{\bar
L}$.

Since $d(\bar{z},\partial\Omega)\geq \varrho$ and
$\bar{z}\notin\Omega$, using \eqref{4.5}, we get $|h(z_j,z_j)|\le
C^\prime\ln\frac1{\varrho}$ for any $i, j$, where $C^\prime>0$.
Thus, we have

\[
\bar G(z_i,z_j)\ge
\ln\frac1{|z_i-z_j|}-C^\prime\ln\frac1{\varrho}\ge \bar
L\ln\frac1{\varrho}-C^\prime\ln\frac1{\varrho}.
\]
Then, by Lemma \ref{l42} and Proposition \ref{ap1}, we have
\[
K(Z)\le \frac{C \de^2
 }{\ln \frac{R}{\ep}
 }+ \frac{ mC^\prime \de^2\ln\frac1\varrho
 }{(\ln\frac{R}{\ep})^2}-\frac{ \bar L \de^2\ln\frac1\varrho\
 }{(\ln\frac{R}{\ep})^2}<c_{\de,1},
 \]
if $\bar L>0$ is large.\\
  \item  Suppose that  there is $i$, such that $z_i\in\partial\Omega_\varrho$.

Let $n$ be the outward unit normal of $\partial\Omega_\varrho$ at
$z_i$. We have

\[
\frac{\partial \bar G(z_j,z_i)}{\partial
n}=-\frac{1}{|z_j-z_i|}\left\langle
\frac{z_i-z_j}{|z_i-z_j|},n\right\rangle-\frac{\partial
g(z_j,z_i)}{\partial n},
\]
where $n$ is the outward normal unit vector of
$\partial\Omega_\varrho$ at $z_i$.

On the other hand, if $z_j\in\Omega_\varrho$, $j\ne i$,  satisfies

\[
\left\langle \frac{z_i-z_j}{|z_i-z_j|},n\right\rangle<0,
\]
then,

\[
\left\langle \frac{z_i-z_j}{|z_i-z_j|},n\right\rangle= O(|z_i-z_j|).
\]
So, we obtain

\[
\left\langle \frac{z_i-z_j}{|z_i-z_j|},n\right\rangle\ge
-C|z_i-z_j|, \quad\forall\; j\ne i.
\]
As a result, by Lemma \ref{l43} and Proposition \ref{ap2}, we have

\[
\begin{split}
\frac{\partial K}{\partial n}\ge &
\frac{4\pi^2\de^2\kappa_i}{|\ln\ep||\ln\frac{R}{\ep}|}\frac{\partial
q(z_i)}{\partial n}+
\frac{2\pi\de^2\kappa_i^2}{(\ln\frac{R}{\ep})^2}\frac{\partial
g(z_i,z_i)}{\partial n}\\
 &+\sum_{j\ne i}^m \frac{2\pi\de^2\kappa_i\kappa_j}{(\ln\frac{R}{\ep})^2}\frac{\partial  g(z_j,z_i)}{\partial
 n}-\frac{C
\de^2}{(\ln\frac{R}{\ep})^2}.
\end{split}
\]
On the other hand,  we derive from \eqref{4.5}

\[
\frac{\partial g(z_i,z_i)}{\partial n}=\frac{1+o(1)}{2\varrho},
\]
and

\[
\frac{\partial g(z_j,z_i)}{\partial
 n}=\frac{1+o(1)}{|\bar z_i-z_j|}\bigl\langle \frac{\bar z_i-z_j}{|\bar
 z_i-z_j|},n\bigr\rangle,
 \]
 where $\bar z_i$ is the reflection point of $z_i$ with respect to
 $\partial\Omega$.

It is easy to check that if $|z_j-z_i|\le M\varrho$, where $M>0$ is
a fixed large constant, then

\[
\left\langle \frac{\bar z_i-z_j}{|\bar
 z_i-z_j|},n\right\rangle\ge 0.
 \]
So

\[
\frac{\partial K}{\partial n}\ge \frac{2\pi
\de^2}{(\ln\frac{R}{\ep})^2}\Bigl(
\frac{\kappa_i^2+o(1)}{2\varrho}-\frac{\kappa_i\kappa_j+o(1)}{M\varrho}-C\Bigr)>0.
\]
\end{enumerate}

Therefore, the flow does not leave $D$.
\end{proof}

\begin{proof}[Proof of Theorem~\ref{th1}]

We will prove that $K(Z)$ has a critical point in
$K^{c_{\de,2}}\setminus K^{c_{\de,1}}$.

Suppose that $K(Z)$ has no critical point in $K^{c_{\de,2}}\setminus
K^{c_{\de,1}}$. Then from Lemma~\ref{l44} that $K^{c_{\de,1}}$ is a
deformation retract of $K^{c_{\de,2}}$.

It is easy to see that  $K^{c_{\de,2}}=D$ and

\[
\bigl\{Z:  Z\in D,  |z_i-z_j|=\varrho^{\bar L}, \;\text{for some}\;
i\ne j\bigr\}\subset K^{c_{\de,1}}.
\]

On the other hand, take $R$ large enough such that
$\inf\limits_{\Omega}q\ge
-\sum_{i=1}^m\frac{(p-1)\kappa_i^2}{16m\pi\min_l\{\kappa_l\}}-\sum_{i=1}^m\frac{\kappa_i^2g(z_j,z_j)}{4m\pi\min_l\{\kappa_l\}}$,
then     $K(Z)\le c_{\de,1}$ implies that

\[
-\sum_{j\ne i}^m \frac{\pi\de^2\kappa_i\kappa_j \bar
G(z_j,z_i)}{(\ln\frac{R}{\ep})^2}\le
-\frac{\pi\de^2\ln\frac1{\varrho}}{|\ln\ep||\ln\frac{R}{\ep}|},
\]
which implies that there are $i\ne j$, such that

\[
\bar G(z_j,z_i)\ge c'\ln\frac1{\varrho},\quad\text{where}~~c'>0
~~\text{is a constant}.
\]
So, there is a $\alpha>0$, independent of $\delta$, such that

\[
|z_i-z_j|\le \varrho^\alpha.
\]
Therefore,

\begin{equation}\label{4.6}
\begin{split}
& \bigl\{Z:  Z\in D,  |z_i-z_j|=\varrho^{\bar L}, \;\text{for
some}\; i\ne
j\bigr\}\\
&\subset K^{c_{\de,1}}\subset \bigl\{Z:  Z\in D, |z_i-z_j|\le
\varrho^\alpha, \;\text{for some}\; i\ne j\bigr\}.
\end{split}
\end{equation}
Filling the hole $D^*=: \bigl\{Z:  Z\in D,  |z_i-z_j|=\varrho^{\bar
L},
 \;\text{for some}\; i\ne
j\bigr\}$ in $D$, we obtain

\begin{equation}\label{4.7}
\begin{split}
& \bigl\{Z:  z_i\in\Omega_\varrho,
 |z_i-z_j|\le \varrho^{\bar L}, \;\text{for some}\; i\ne
j\bigr\}\\
&\subset K^{c_{\de,1}}\cup D^* \subset \bigl\{Z:
z_i\in\Omega_\varrho,
 |z_i-z_j|\le
\varrho^\alpha, \;\text{for some}\; i\ne j\bigr\}.
\end{split}
\end{equation}

Since $K^{c_{\de,1}}$ is a deformation retract of $K^{c_{\de,2}}$,
we find that $K^{c_{\de,1}}\cup D^*$ is a deformation retract of
$K^{c_{\de,2}} \cup D^*$. On the other hand,
 $\bigl\{Z:  z_i\in\Omega_\varrho,  z_i=z_j, \;\text{for some}\; i\ne
j\bigr\}$ is a  deformation retract of $\bigl\{Z:
z_i\in\Omega_\varrho,
 |z_i-z_j|\le
\varrho^\alpha, \;\text{for some}\; i\ne j\bigr\}$ if $\varrho>0$ is
small. Using \eqref{4.7}, we see that

\[
\bigl\{Z:  z_i\in\Omega_\varrho,  z_i=z_j, \;\text{for some}\; i\ne
j\bigr\}
\]
is a deformation retract of
\[
\underbrace{\Omega_\varrho\times\cdots\times \Omega_\varrho}_m=
K^{c_{\de,2}} \cup D^*.
\]
This is impossible if $\Omega$ has nontrivial homology.

Thus we get a solution $w_\de$ for \eqref{1}. Let
$u_\ep=\frac{|\ln\ep|}{2\pi}w_\de,~\de=\ep\left(\frac{|\ln\ep|}{2\pi}\right)^{\frac{1-p}{2}}$,
it is not difficult to check that  $u_\ep$ has all the properties
listed in Theorem \ref{th1} and thus the proof of Theorem \ref{th1}
is complete.

\end{proof}

\begin{remark}\label{r44}
In the proof of Theorem \ref{th1}, what we actually need is that  the
following function
\[
\Phi(Z)=\sum_{i=1}^m4\pi^2\kappa_iq(z_i) +\sum_{i=1}^m
\pi\kappa_i^2g(z_i,z_i)-\sum_{j\neq
i}^m\pi\kappa_i\kappa_j\bar{G}(z_j,z_i)
\]
as well as its small perturbation (in a suitable sense) has a
critical point in $D$. Moreover, using the  estimates as in Lemma
\ref{l43}, it is easy to see that if
$\sum_{j=1}^mP_{\de,Z_{\de},j}(x)+\omega_\de$ is a solution of
\eqref{1}, and $Z_\de\rightarrow Z_0$ as $\de\rightarrow0$, then
$Z_0$ is a critical point of $\Phi(Z)$.
\end{remark}

\begin{proof}[Proof of Theorem~\ref{th3}]
Note that the Kirchhoff--Routh function associated to the vortex
dynamics is
$$\mathcal{W}(x_1,\cdots,x_m)=\frac{1}{2}\sum_{i\neq
j}^m\kappa_i\kappa _jG(x_i,x_j)+\frac{1}{2}\sum^{m}_{i=1}\kappa^2
_iH(x_i,x_i)+\sum^{m}_{i=1}\kappa_i\psi_0(x_i). $$

 Recall that
$h(z_i,z_j)=-H(z_i,z_j)$, it is easy to check that
\[
\Phi(Z)=-4\pi^2\mathcal {W}(Z)+\sum_{i=1}^m\pi\kappa_i^2\ln R.
\]
Hence, $\Phi(Z)$ and $\mathcal{W}(Z)$ possess the same critical
points.

By Lemma \ref{l42}, \ref{l43} and Proposition \ref{ap1}, \ref{ap2},
we have
 \[
 K(Z)= \frac{ C\de^2
 }{\ln\frac{R}{\ep}
 } +\sum_{i=1}^m\frac{\pi(p-1)\de^2\kappa_i^2}{4(\ln\frac{R}{\ep})^2}+\frac{\de^2}{|\ln\ep|^2}\Phi(Z)
+ O\left(\frac{\de^2\ln|\ln\ep|}{|\ln\ep|^3} \right)
\]
and
\[
\frac{\partial K(Z)}{\partial
z_{i,h}}=\frac{\de^2}{|\ln\ep|^2}\frac{\partial \Phi(Z)}{\partial
z_{i,h}}+O\left(\frac{\de^2\ln|\ln\ep|}{|\ln\ep|^3} \right).
\]
Thus, stable critical point of Kirchhoff-Routh function
$\mathcal{W}(Z)$ implies that  $K(Z)$ has a critical point. So the
result follows.
\end{proof}

Now we are in the position to prove Theorem~\ref{nth1}.
\begin{proof}[Proof of Theorem~\ref{nth1}]

By Theorem \ref{th3}, we obtain that $u_\ep$ is a solution to
\eqref{0}.

Set
 \[
\mathbf{v}_\ep=(\nabla
(u_\ep-q))^\bot,~\omega_\ep=\nabla\times\mathbf{v}_\ep,
\]
\[
P_\ep=\sum_{j=1}^m\frac{1}{p+1}\chi_{\Om_j}\left(u_\ep-q-\frac{\kappa_j|\ln\ep|}{2\pi}\right)_+^{p+1}-\frac{1}{2}|\nabla
(u_\ep-q)|^2.
\]
then  $(\mathbf{v}_\ep, P_\ep)$ forms a stationary solution for
problem \eqref{1.2}.

We now just need to verify
\[ \int_\Omega
\omega_\ep\rightarrow \sum_{j=1}^m\kappa_j,
~~\text{as}~~\ep\rightarrow0.
\]

By direct calculations, we find that
\[
\begin{split}
\int_\Omega\omega_\ep&=\sum_{j=1}^m\frac{1}{\ep^2}\int_\Omega\chi_{\Om_j}
\left(u_\ep-q-\frac{\kappa_j|\ln\ep|}{2\pi}\right)_+^p\\
&=\sum_{j=1}^m\frac{|\ln\ep|^p}{(2\pi)^p \ep^2}\int_{\Omega_j}
\left(w_\de-\kappa_j-\frac{2\pi q}{|\ln\ep|}\right)^p_+\\
&=\sum_{j=1}^m\frac{|\ln\ep|^p}{(2\pi)^p\ep^2}\int_{B_{Ls_{\de,j}(z_j)}}
\left(W_{\de,z_j,a_{\de,j}}-a_{\de,j}+O\Big(\frac{s_{\de,j}}{|\ln\ep|}\Big)\right)^p_+\\
&=\sum_{j=1}^m\frac{ s_{\de,j}^2|\ln\ep|^p}{(2\pi)^p\ep^2}\left(\frac{\de}{s_{\de,j}}\right)^{\frac{2p}{p-1}}\int_{B_1(0)}\phi^p+o(1)\\
&=\sum_{j=1}^m\frac{ a_{\de,j}
|\ln\ep|}{\ln\frac{R}{s_{\de,j}}}+o(1)\\
&\rightarrow \sum_{j=1}^m\kappa_j, \quad \text{as}~~\ep\rightarrow0.
\end{split}
\]
Therefore, the result follows.
\end{proof}

\begin{remark}
To regularize point vortices with equi-strength $\kappa$, we  do not
need $\chi_{\Om_j}$, that is, we just need to consider the following
problem
\[
 \begin{cases}
-\ep^2 \Delta u=(u-q-\frac{\kappa}{2\pi}\ln\frac{1}{\ep})_+^p, \quad & x\in\Omega, \\
u=0, \quad & x\in\partial\Omega.
\end{cases}
\]

\end{remark}

{\bf Acknowledgements:} D. Cao and Z. Liu were supported by the
National Center for Mathematics and Interdisciplinary Sciences, CAS.
D. Cao and J. Wei were also supported by CAS Croucher Joint
Laboratories Funding Scheme.

\appendix

\section{Energy expansion}

In this section we will give precise expansions of
$I\left(\sum_{j=1}^m P_{\de,Z,j}\right)$ and $\frac{\partial
}{\partial z_{i,h}} I\left(\sum_{j=1}^m P_{\de,Z,j}\right)$, which
have been used in section 4.

%Let

%\begin{equation}\label{a11}
%I(u)=\frac{\de^2} 2 \int_\Om |Du|^2 -\frac1 {p+1} \int_\Om
%\left(u-1-\frac{2\pi q(x)}{\kappa|\ln\ep|}\right)_+^{p}.
%\end{equation}
%In this section, we will expand
%$I\bigl(\sum_{j=1}^mP_{\de,Z,j}\bigr)$.

We always assume that

\[
d(z_j,\partial\Om)\ge \varrho>0,~~~|z_i-z_j|\ge \varrho^{\bar
L},\quad i\ne j,
\]
for some small $\varrho>0$ and large $\bar L>0$.

\begin{lemma}\label{al1}
For $x\in\Om_i,~i=1,\cdots,m$, we have
\[
P_{\de,Z}(x)>\kappa_i+\frac{2\pi q(x)}{|\ln\ep|},\quad x\in
B_{s_{\de,i}(1-Ts_{\de,i})}(z_i),\;
\]
where $T>0$ is a large constant; while

\[
P_{\de,Z}(x)<\kappa_i+\frac{2\pi q(x)}{|\ln\ep|},\quad
x\in\Om_i\setminus B_{s_{\de,i}(1+s_{\de,i}^{\sigma})}(z_i),
\]
where $\sigma>0$ is a small constant.
\end{lemma}

\begin{proof}

Suppose that $x\in B_{s_{\de,i}(1-Ts_{\de,i})}(z_i)$. It follows
from \eqref{2.9} and $\phi_1'(s)<0$ that
\[
\begin{split}
& P_{\de,Z}(x)-\kappa_i-\frac{2\pi q(x)}{|\ln\ep|}=  W_{\de,z_i,a_{\de,i}}(x)-a_{\de,i}+O\left(\frac{s_{\de,i}}{|\ln\ep|}\right)\\
=&
\frac{a_{\de,i}}{|\phi'(1)||\ln\frac{R}{s_{\de,i}}|}\phi\Bigl(\frac{|x-z_i|}{s_{\de,i}}\Bigr)+O\Bigl(\frac{\ep}{|\ln\ep|}\Bigr)>0,
\end{split}
\]
if $T>0$ is large.   On the other hand, if $x\in\Om_i\setminus
 B_{s_{\de,i}^{\tilde\sigma}}(z_i)$, where
$\tilde\sigma>\sigma>0$ is a fixed small constant,  then
\[
\begin{split}
& P_{\de,Z}(x)-\kappa_i-\frac{2\pi q(x)}{|\ln\ep|}= \sum_{j=1}^m
a_{\de,j}\ln\frac R{|x-z_j|}/\ln\frac R{s_{\de,j}} -\kappa_i
-\frac{2\pi q(x)}{|\ln\ep|}+o(1)\\
\le &  C\tilde \sigma -\kappa_i+o(1)<0.
\end{split}
\]

Finally,  if $x\in B_{s_{\de,i}^{\tilde\sigma}}(z_i)\setminus
B_{s_{\de,i}(1+Ts_{\de,i}^{\tilde \sigma})}(z_i)$ for some $i$, then
\[
\begin{split}
& P_{\de,Z}(x)-\kappa_i-\frac{2\pi q(x)}{|\ln\ep|}=  W_{\de,z_i,a_{\de,i}}(x)-a_{\de,i}+O\left(\frac{s_{\de,i}^{\tilde\sigma}}{\ln\frac{R}{s_{\de,i}}}\right)\\
=& a_{\de,i}\frac{\ln\frac
R{|x-z_i|}}{\ln\frac{R}{s_{\de,i}}}-a_{\de,i}+O\left(\frac{s_{\de,i}^{\tilde\sigma}}{\ln\frac{R}{s_{\de,i}}}\right)\\
\le & -a_{\de,i}\frac{\ln(1+Ts_{\de,i}^{\tilde\sigma})}{\ln\frac
{R}{s_{\de,i}}}+O\left(\frac{s_{\de,i}^{\tilde\sigma}}{\ln\frac{R}{s_{\de,i}}}\right)
 <0,
\end{split}
\]
if $T>0$ is large. Note that by the choice of $\tilde\sigma$,
$B_{s_{\de,i}(1+s_{\de,i}^\sigma)}(z_i)\supset
B_{s_{\de,i}(1+Ts_{\de,i}^{\tilde\sigma})}(z_i)$ for small $\de$. We
therefore derive our conclusion.
\end{proof}

\begin{proposition}\label{ap1}

We have

\[
\begin{split}
I\left( P_{\de,Z}\right)=&\frac{ C\de^2
 }{\ln\frac{R}{\ep}
 } +\sum_{i=1}^m\frac{\pi(p-1)\de^2\kappa_i^2}{4(\ln\frac{R}{\ep})^2}
 +\sum_{i=1}^m\frac{4\pi^2\de^2\kappa_iq(z_i)}{|\ln\ep||\ln\frac{R}{\ep}|}+
 \sum_{i=1}^m \frac{\pi\de^2\kappa_i^2 g(z_i,z_i)}{(\ln\frac{R}\ep)^2}
 \\
 &-\sum_{j\ne i}^m \frac{\pi\de^2\kappa_i\kappa_j \bar  G(z_j,z_i)}{{(\ln\frac{R}\ep)^2}}
+ O\left(\frac{\de^2\ln|\ln\ep|}{|\ln\ep|^3} \right),
 \end{split}
\]
where $C$ is a positive constant.

\end{proposition}

\begin{proof}

Taking advantage of  \eqref{2.3}, we have

\[
\begin{split}
&\de^2\int_\Om \big|D P_{\de,Z}\big|^2= \sum_{j=1}^m \sum_{i=1}^m
\int_\Om \bigl(W_{\de,z_j,a_{\de,j} }-a_{\de,j}\bigr)_+^{p}
P_{\de,Z,i}\\
=& \sum_{j=1}^m \sum_{i=1}^m\int_{B_{s_{\de,j}}(z_j)}
\left(W_{\de,z_j,a_{\de,j}}-a_{\de,j}\right)_+^{p} \left(
W_{\de,z_i,a_{\de,i}}-\frac{a_{\de,i}}{\ln\frac R{s_{\de,i}}}
g(x,z_i)\right).
\end{split}
\]

First, we estimate

\[
\begin{split}
&\int_{B_{s_{\de,i}} (z_i)}
\left(W_{\de,z_i,a_{\de,i}}-a_{\de,i}\right)_+^{p}\left(
W_{\de,z_i,a_{\de,i}}-\frac{a_{\de,i}}{\ln\frac R{s_{\de,i}}}
g(x,z_i)
\right)\\
=&\int_{B_{s_{\de,i}} (z_i)}
\bigl(W_{\de,z_i,a_{\de,i}}-a_{\de,i}\bigr)^{p+1}+a_{\de,i}
\int_{B_{s_{\de,i}} (z_i)}
\bigl(W_{\de,z_i,a_{\de,i}}-a_{\de,i}\bigr)^{p}\\
&-\frac{a_{\de,i}}{\ln\frac R{s_{\de,i}}}\int_{B_{s_{\de,i}} (z_i)}
\bigl(W_{\de,z_i,a_{\de,i}}-a_{\de,i}\bigr)^{p}
g(x,z_i)\\
=&\Bigl(\frac{\de}{s_{\de,i}}\Bigr)^{\frac{2(p+1)}{p-1}}s_{\de,i}^2\int_{B_1(0)}\phi^{p+1}
+a_{\de,i}\Bigl(\frac{\de}{s_{\de,i}}\Bigr)^{\frac{2p}{p-1}}s_{\de,i}^2\int_{B_1(0)}\phi^{p}\\
&-\frac{a_{\de,i}}{\ln\frac
R{s_{\de,i}}}\Bigl(\frac{\de}{s_{\de,i}}\Bigr)^{\frac{2p}{p-1}}g(z_i,z_i)s_{\de,i}^2\int_{B_1(0)}\phi^{p}+O\left(
\frac{s_{\de,i}^3}{|\ln\ep|^{p+1}}\right)\\
=&\frac{\pi(p+1)}{2}\frac{\de^2a_{\de,i}^2}{(\ln\frac{R}{s_{\de,i}})^2}+\frac{2\pi
\de^2a_{\de,i}^2}{\ln\frac{R}{s_{\de,i}}}-\frac{2\pi\de^2
a_{\de,i}^2}{(\ln\frac{R}{s_{\de,i}})^2}g(z_i,z_i)+O\left(
\frac{\ep^3}{|\ln\ep|^{p+1}}\right).
\end{split}
\]

Next, for  $j\ne i$,

\[
\begin{split}
&\int_{B_{s_{\de,j}} (z_j)}
\bigl(W_{\de,z_j,a_{\de,j}}-a_{\de,j}\bigr)_+^{p} \left(
W_{\de,z_i,a_{\de,i}}-\frac{a_{\de,i}}{\ln\frac R{s_{\de,i}}}
g(x,z_i)
\right)\\
=&
\Bigl(\frac{\de}{s_{\de,j}}\Bigr)^{\frac{2p}{p-1}}\frac{a_{\de,i}}{\ln\frac{R}{s_{\de,i}}}
\int_{B_{s_{\de,j}}(z_j)}\phi^{p}\Bigl(\frac{|x-z_j|}{s_{\de,j}}\Bigr)\bar
G(x,z_i)
\\
=&
\Bigl(\frac{\de}{s_{\de,j}}\Bigr)^{\frac{2p}{p-1}}\frac{a_{\de,i}s_{\de,j}^2}{\ln\frac{R}{s_{\de,i}}}
\int_{B_{1}(0)}\phi^{p}(|x|)\bar G(z_j+s_{\de,j}x,z_i)
\\
=&\Bigl(\frac{\de}{s_{\de,j}}\Bigr)^{\frac{2p}{p-1}}\frac{a_{\de,i}s_{\de,j}^2}{\ln\frac{R}{s_{\de,i}}}
\bar G(z_j,z_i)
\int_{B_{1}(0)}\phi^{p}+O\left(\frac{s_{\de,j}^3}{|\ln\ep|^{p+1}}\right)\\
=&\frac{2\pi\de^2
a_{\de,i}a_{\de,j}}{|\ln\frac{R}{s_{\de,i}}||\ln\frac{R}{s_{\de,j}}|}\bar{G}(z_i,z_j)+O\left(
\frac{\ep^3}{|\ln\ep|^{p+1}}\right).
\end{split}
\]

By Lemma~\ref{al1} and \eqref{2.9},

\[
\begin{split}
&\sum_{j=1}^m\int_{\Om_j} \left( P_{\de,Z} -\kappa_j-\frac{2\pi
q(x)}{|\ln\ep|}\right)_+^{p+1}= \sum_{j=1}^m
\int_{B_{Ls_{\de,j}}(z_j)}\left( P_{\de,Z} -\kappa_j-\frac{2\pi
q(x)}{\kappa|\ln\ep|}\right)_+^{p+1}
\\
=& \sum_{j=1}^m \int_{B_{Ls_{\de,j}}(z_j)}
\left(W_{\de,z_j,a_{\de,j}} -a_{\de,j} +O\bigg(\frac
{s_{\de,j}}{|\ln\ep|}
\bigg)\right)_+^{p+1}\\
=&\sum_{j=1}^m\left(\frac{\de}{s_{\de,j}}\right)^{\frac{2(p+1)}{p-1}}\int_{B_{s_{\de,j}}(z_j)}\phi^{p+1}\Bigl(\frac{|x-z_j|}{s_{\de,j}}\Bigr)
+O\left(\frac{s_{\de,j}^3}{|\ln\ep|^{p+1}}\right)\\
=&\sum_{j=1}^m\left(\frac{\de}{s_{\de,j}}\right)^{\frac{2(p+1)}{p-1}}s_{\de,j}^2\int_{B_{1}(0)}\phi^{p+1}
+O\left(\frac{s_{\de,j}^3}{|\ln\ep|^{p+1}}\right)\\
=&\sum_{j=1}^m\frac{\pi(p+1)}{2}\frac{\de^2a_{\de,j}^2}{(\ln\frac{R}{s_{\de,j}})^2}+O\left(
\frac{\ep^3}{|\ln\ep|^{p+1}}\right).
\end{split}
\]
So, we have proved

\[
\begin{split}
I\left(\sum_{i=1}^m P_{\de,Z,j}\right)=&\sum_{i=1}^m\left(\frac{\pi
(p+1)}{4} \frac{ \de^2
 a_{\de,i}^2}{|\ln \frac{R}{s_{\de,i}}|^2
 } +\frac{\pi\de^2 a^2_{\de,i}}{|\ln \frac{R}{s_{\de,i}}|}-\frac{\pi g(z_i,z_i)\de^2a_{\de,i}^2}{|\ln\frac{R}{s_{\de,i}}|^2}\right)
 \\
 &+\sum_{j\ne i}^m \frac{\pi\bar G(z_j,z_i)\de^2 a_{\de,i}
 a_{\de,j}}{|\ln\frac{R}{s_{\de,i}}||\ln \frac{R}{s_{\de,j}}|}-\frac{\pi\de^2}{2}\left(\sum_{j=1}^m \frac{a_{\de,j}^2}{|\ln \frac{R}{s_{\de,j}}|^2}\right)
+ O\left(\frac{\ep^3}{|\ln\ep|^{p+1}} \right).
 \end{split}
\]

Thus, the result follows from Remark~\ref{remark2.2}.

\end{proof}

\begin{proposition}\label{ap2}

We have

\[
\begin{split}
\frac{\partial }{\partial z_{i,h}} I\left(
P_{\de,Z}\right)=&\frac{4\pi^2\de^2\kappa_i}{|\ln\ep||\ln\frac{R}{\ep}|}\frac{\partial
q(z_i)}{\partial z_{i,h}}
+\frac{2\pi\de^2\kappa_i^2}{(\ln\frac{R}{\ep})^2} \frac{\partial
g(z_i,z_i)}{\partial z_{i,h}}
 -\sum_{j\ne i}^m \frac{2\pi\de^2\kappa_i\kappa_j}{(\ln\frac{R}{\ep})^2}\frac{\partial \bar G(z_j,z_i)}{\partial z_{i,h}} \\
&+ O\left(\frac{\de^2\ln|\ln\ep|}{|\ln\ep|^3} \right).
 \end{split}
\]

\end{proposition}

\begin{proof}

Direct computation yields that

\[
\begin{split}
\frac{\partial }{\partial z_{i,h}} I\left( P_{\de,Z}\right)=&
\de^2\int_{\Om} DP_{\de,Z}D \frac{\partial P_{\de,Z}}{\partial
z_{i,h}}-\sum_{j=1}^m\int_{\Om_j} \left( P_{\de,Z}-\kappa_j
-\frac{2\pi q(x)}{|\ln\ep|}\right)_+^{p}\frac{\partial
P_{\de,Z}}{\partial z_{i,h}}
\\
=& \sum_{l=1}^m\sum_{j=1}^m \int_{B_{Ls_{\de,j}}(z_j)}
 \left[ \left(W_{\de,z_j,a_{\de,j}}-a_{\de,j}\right)_+^{p}-
 \left(
P_{\de,Z}-\kappa_j -\frac{2\pi
q(x)}{|\ln\ep|}\right)_+^{p}\right]\frac{\partial
P_{\de,Z,l}}{\partial z_{i,h}} .
 \end{split}
\]

Using \eqref{2.9}, Lemma~\ref{al1} and Remark~\ref{remark2.2}, we
find that

\[
\begin{split}
&\int_{B_{Ls_{\de,i}}(z_i)}
 \left[ \left(W_{\de,z_i,a_{\de,i}}-a_{\de,i}\right)_+^{p}-
 \left(
P_{\de,Z}-\kappa_i-\frac{2\pi q(x)}{|\ln\ep|} \right)_+^{p}\right]\frac{\partial P_{\de,Z,i}}{\partial z_{i,h}}\\
=&\int_{B_{s_{\de,i}(1+s_{\de,i}^{\sigma})}(z_i)}
  \left[ \bigl(W_{\de,z_i,a_{\de,i}}-a_{\de,i}\bigr)_+^{p}-
  \left(
P_{\de,Z}-\kappa_i-\frac{2\pi q(x)}{|\ln\ep|} \right)_+^{p}\right]\frac{\partial P_{\de,Z,i}}{\partial z_{i,h}}\\
=
&p\int_{B_{s_{\de,i}}(z_i)}\bigl(W_{\de,z_i,a_{\de,i}}-a_{\de,i}\bigr)_+^{p-1}\bigg[
 \frac{2\pi}{|\ln\ep|}\bigl\langle D q(z_i), x-z_i\bigr\rangle+\frac{a_{\de,i}}{\ln\frac{R}{s_{\de,i}}}\bigl\langle D g(z_i,z_i), x-z_i\bigr\rangle \\
\quad&- \sum_{j\ne i}^m
\frac{a_{\de,j}}{\ln\frac{R}{s_{\de,j}}}\bigl\langle D \bar
G(z_i,z_j),x-z_i\bigr\rangle \bigg]\frac{\partial
P_{\de,Z,i}}{\partial z_{i,h}}+O\Bigl(\frac{\ep^{2+\sigma}}{|\ln\ep|^{p+1}}\Bigr)\\
=&-\frac{p\de^2a_{\de,i}}{|\phi'(1)||\ln\frac{R}{s_{\de,i}}|}\bigg(\frac{2\pi}{|\ln\ep|}\frac{\partial
q(z_i)}{\partial
z_{i,h}}+\frac{a_{\de,i}}{\ln\frac{R}{s_{\de,i}}}\frac{\partial
g(z_i,z_i)}{\partial z_{i,h}}-\sum_{j\neq
i}^m\frac{a_{\de,j}}{\ln\frac{R}{s_{\de,j}}}\frac{\partial \bar
G(z_i,z_j)}{\partial z_{i,h}}\bigg)\\
 &\times \int_{B_1(0)}\phi^{p-1}(|x|)\phi^\prime(|x|)\frac{x_h^2}{|x|}
 +O\Bigl(\frac{\ep^{2+\sigma}}{|\ln\ep|^{p+1}}\Bigr)\\
 =&\frac{4\pi^2 \de^2a_{\de,i}}{|\ln\ep||\ln\frac{R}{s_{\de,i}}|}\frac{\partial q(z_i)}{\partial z_{i,h}}
 +\frac{2\pi \de^2a_{\de,i}^2}{(\ln\frac{R}{s_{\de,i}})^2}\frac{\partial g(z_i,z_i)}{\partial z_{i,h}}
 -\sum_{j\neq i}\frac{2\pi \de^2a_{\de,i}a_{\de,j}}{|\ln\frac{R}{s_{\de,j}}||\ln\frac{R}{s_{\de,i}}|}\frac{\partial \bar G(z_i,z_j)}{\partial z_{i,h}}
 +O\Bigl(\frac{\ep^{2+\sigma}}{|\ln\ep|^{p+1}}\Bigr),
\end{split}
\]
since

\[
\int_{B_1(0)}\phi^{p-1}(|x|)\phi^\prime(|x|)\frac{x_h^2}{|x|}=
-\frac{2\pi}{p}|\phi^\prime(1)|.
\]

On the other hand, for $l\ne i$, from \eqref{2.9}, we have

\[
\begin{split}
&\int_{B_{Ls_{\de,j}}(z_j)}
 \left[ \bigl(W_{\de,z_j,a_{\de,j}}-a_{\de,j}\bigr)_+^{p}-
 \left(
P_{\de,Z}-\kappa_j-\frac{2\pi q(x)}{|\ln\ep|} \right)_+^{p}\right]\frac{\partial P_{\de,Z,l}}{\partial z_{i,h}}\\
= &\int_{B_{Ls_{\de,j}}(z_j)}\left[\left(W_{\de,z_j,a_{\de,j}}-
a_{\de,j}\right)^{p-1}\frac{s_{\de,j}}{|\ln\ep|}\right]\times \frac{C}{\ln\frac{R}{s_{\de,l}}}\\
 =& O\left(\frac{\ep^3}{|\ln\ep|^{p+1}}
\right).
\end{split}
\]
Thus, the result follows.

\end{proof}

\end{document}